\documentclass{amsart}
\usepackage{amsmath, amssymb, amsthm, url}

\input xy
\xyoption{all}

\newtheorem{thm}{Theorem}[section]
\newtheorem{lem}[thm]{Lemma}
\newtheorem{cor}[thm]{Corollary}

\newtheorem{prop}[thm]{Proposition}

\newtheorem{prob}[thm]{Problem}

\theoremstyle{definition}

\newtheorem{rmk}[thm]{Remark}
\newtheorem{defn}[thm]{Definition}
\newtheorem{notation}[thm]{Notation}

\newtheorem*{motiv}{Motivation}

\def\L{\mathcal{L}}
\def\M{\mathcal{M}}
\def\mN{\mathcal{N}}

\def\N{\mathbb{N}}
\def\Q{\mathbb{Q}}
\def\R{\mathbb{R}}
\def\Z{\mathbb{Z}}

\def\tp{\operatorname{tp}}
\def\qftp{\operatorname{qftp}}
\def\id{\operatorname{id}}
\def\alg{\operatorname{alg}}
\def\LO{\operatorname{LO}}
\def\DLO{\operatorname{DLO}}
\def\OAG{\operatorname{OAG}}
\def\ODAG{\operatorname{ODAG}}
\def\MODAG{\operatorname{MODAG}}
\def\MODDAG{\operatorname{div-MODAG}}
\def\OF{\operatorname{OF}}
\def\RCF{\operatorname{RCF}}
\def\near{\operatorname{near}}
\def\next{\operatorname{next}}

\begin{document}

\title{Model companion of ordered theories with an automorphism}

\author{Michael C. Laskowski}
\address{Department of Mathematics, University of Maryland, College Park, MD 20742}
\email{mcl@math.umd.edu}

\author{Koushik Pal}
\address{Department of Mathematics, University of Maryland, College Park, MD 20742}
\email{koushik@math.umd.edu}

\thanks{The authors are partially supported by Laskowski's NSF grant DMS-0901336.}

\subjclass[2000]{Primary 03C10, 03C64; Secondary 20K30, 20A05}

\date{Submitted May 31, 2013.}

\begin{abstract}
Kikyo and Shelah showed that if $T$ is a theory with the Strict Order Property in some first-order language $\L$, then in the expanded language $\L_\sigma := \L\cup\{\sigma\}$ with a new unary function symbol $\sigma$, the bigger theory $T_\sigma := T\cup\{``\sigma \mbox{ is an } \L\mbox{-automorphism''}\}$ does not have a model companion. We show in this paper that if, however, we restrict the automorphism and consider the theory $T_\sigma$ as the base theory $T$ together with a ``restricted'' class of automorphisms, then $T_\sigma$ can have a model companion in $\L_\sigma$. We show this in the context of linear orders and ordered abelian groups.
\end{abstract}

\maketitle

\section{Introduction}
A major development in the model theory of fields occurred when Chatzidakis and Hrushovski \cite{ZCEH}, and independently Macintyre \cite{AM}, showed that the theory of algebraically closed fields equipped with a generic automorphism has a model companion, namely the theory ACFA.  Shortly thereafter, a number of researchers tried to extend this result to more general theories.  Specifically, if $T$ is a model complete theory in a language $\L$, let $\L_\sigma$ be the expansion of $\L$  formed by adding a new unary function symbol $\sigma$, and let $T_\sigma$ be the theory of a generic automorphism. That is, $T_\sigma$ is the expansion of $T$ formed by adding axioms asserting that $\sigma$ describes an $\L$-automorphism, but with no other constraints. In this framework, one can ask for which theories $T$ does $T_\sigma$ have a model companion? Baldwin and Shelah~\cite{JBSS} gave a precise characterization of which stable theories $T$ have this property. As for unstable theories, a precise characterization is still not known, but Kikyo and Shelah~\cite{HKSS} proved that if the original theory $T$ has the strict order property, then the theory $T_\sigma$ cannot have a model companion in $\L_\sigma$. Some other results of a similar flavor can be found in \cite{HK}, \cite{HKAP} and \cite{ZCAP}.

At first blush, the Kikyo-Shelah result is disappointing, as it appears to rule out a good theory of `difference valued fields' since the theory of any field with a nontrivial valuation has the strict order property via the definable total ordering of the value group. However, in his thesis \cite{KP}, the second-named author noted that if one placed more restrictions on the automorphism, then it is sometimes possible for the  expanded theory to have a model companion.

In this paper, we continue this line of reasoning.  We begin with a theory $T$ that admits an infinite total order on the elements of any model (hence $T$ has the strict order property) and we ask which theories $T'_\sigma$ extending $T_\sigma$ have model companions. We accomplish this by asking the stronger question of identifying  the extensions $S_\sigma$ of $T_\sigma$ that are model complete. Then, if we have a complete list, a theory $T'_\sigma$ extending $T_\sigma$ has a model companion if and only if it has the same universal theory as one of the $S_\sigma$'s.

In Section 2, which is largely a warm up for our main results, we start with the theory of linear orders and add a generic automorphism. Whereas this theory does not have a model companion by Kikyo-Shelah, we show that if we insist that the automorphism is everywhere increasing, then the expanded theory does have a model companion.  Building on this, we give a complete enumeration of the (countably many) complete, model complete theories extending dense linear orders with a generic automorphism.  Along the way, and perhaps the most interesting point, is that we put our finger on an obstruction to model completeness (Theorem 2.6).  Indeed, this obstruction appears to be the phenomenon that is exploited by Kikyo and Shelah to obtain  their negative result.

The main sections of the paper are Sections 3 and 4, where we discuss expansions of ordered abelian groups.  In Section 3, we show that simply asserting that the automorphism is everywhere increasing on the positive elements of the group is not enough to give a model companion. The theory needs to be strong enough to assert more about the automorphism.  In Section 4, we revisit the theories of MODAGs and div-MODAGs that were introduced in the second-named author's thesis.  There, he proved that the theory of any div-MODAG is model complete.  However, we show in this paper that there are many more model complete expansions of divisible ordered abelian groups with an automorphism. We introduce the concept of an `$n$-sum' of div-MODAGs and prove that each of these has a model complete theory. Then, we investigate `$\omega$-sums' of div-MODAGs.  With Theorem~4.28 we prove that an $\omega$-sum has a model complete theory if and only if there is a `unique type at infinity'. It is insightful to note here how this property compares with Theorem 2.6 as to how the Kikyo-Shelah obstruction is eradicated. In the appendix, we give a quotient construction which produces more examples of model complete theories extending the theory of ordered abelian groups with an automorphism. Finally, in Section 5, we deal with the case of ordered fields with an automorphism and briefly discuss some of the difficulties in  obtaining a model complete expansion of $RCF_\sigma$.

In many ways, this paper is an extension of Chapter 3 of the second-named author's PhD thesis, written under the supervision of Thomas Scanlon.  We thank him for his insightful suggestions and helpful discussions on this topic.

\section{Linear Order with Increasing Automorphism}
Let $\LO$ be the theory of linear orders in the language $\L_{<} := \{<\}$, where $<$ is a binary relation symbol. As is well-known, the theory of linear orders has a model companion in $\L_{<}$, namely the theory of nontrivial dense linear orders without endpoints ($\DLO$). The goal is now to consider this structure with an automorphism $\sigma$ and see if one can get a model companion. As noted earlier, $\sigma$ cannot be a generic automorphism if one wants model companion to exist. So we must put some restriction on $\sigma$. Taking a hint from the proof of the Kikyo-Shelah theorem, one natural restriction that one can impose is the following.
\begin{defn}
An automorphism $\sigma$ on $\LO$ is said to be {\em increasing} if 
$$\forall x ( x < \sigma(x) ).$$
\end{defn}

We denote by $\LO^+_\sigma$ the theory of linear orders together with the axioms denoting ``$\sigma$ is an increasing $\L_{<}$-automorphism'' in the language $\L_{<, \sigma} := \{<, \sigma\}$. (There is an analogous theory $\LO^-_\sigma$ with decreasing automorphisms too.) We claim that $LO_\sigma^+$ has a model companion in $\L_{<, \sigma}$, namely, the theory of nontrivial dense linear orders with an increasing automorphism, which we denote by $\DLO^+_\sigma$. 

\begin{thm}
$\DLO^+_\sigma$ eliminates quantifiers and is the model completion of $\LO^+_\sigma$ in the language $\L_{<,\sigma}$. Moreover, $\DLO^+_\sigma$ is complete and o-minimal.
\end{thm}
The proof of this theorem follows a very similar argument as the proof of the corresponding theorem for DLO (without the automorphism) with appropriate modifications to incorporate the automorphism. We leave the details to the reader.

We have thus successfully shown that the Kikyo-Shelah result can be salvaged by putting extra restrictions on the automorphism -- the case of increasing (also, decreasing) automorphism is only one such example. Questions now arise, how many other examples can there be, and can one characterize all such model complete theories. We deal with these questions in the remaining part of this section. We first prove in Theorem 2.6 a general fact about model complete $\L_{<,\sigma}$-theories extending $\LO_\sigma$, which is the main obstruction underlying the Kikyo-Shelah theorem. Then we use this theorem to give a complete characterization of all complete and model complete extensions of the theory of dense linear orders with an automorphism. In particular, we show that there are only countably many such model complete extensions. We start with a few definitions.

Fix an $\L_{<,\sigma}$-structure $\M$.
\begin{defn}
A {\em cut} $C = (A, B)$ is a partition of $M = A\cup B$ into disjoint non-empty sets $A$ and $B$, where $A$ is an initial segment of $\M$ (i.e., $A < B$). 

A cut $C = (A, B)$ is called {\em $\sigma$-closed} if both $A$ and $B$ are closed under $\sigma$.

For a given $\L_{<, \sigma}$-formula $\varphi(x)$, a cut $C = (A, B)$ is called {\em $\varphi$-split} if there are $a^*\in A$ and $b^*\in B$ such that $\M\models\varphi(a)$ for all $a\in A$ with $a\ge a^*$ and $\M\models\neg\varphi(b)$ for all $b\in B$ with $b\le b^*$.

A cut $C = (A, B)$ is called {\em split} if $C$ is $\varphi$-split for some $\L$-formula $\varphi(x)$.
\end{defn}

\begin{defn}
Let $x, y\in M$. We say $x$ and $y$ are in the same $\sigma$-{\em archimedian class} if there are $m, n\in\Z$ such that $\sigma^m(x) \le y \le \sigma^n(x)$.
\end{defn}

\begin{rmk}
Note that a $\sigma$-archimedian class is closed under $\sigma$ and $\sigma^{-1}$. Moreover, if $C = (A, B)$ is a $\sigma$-closed cut, then for any $\sigma$-archimedian class $D$, either $D\subseteq A$ or $D\subseteq B$. Finally, if $C = (A, B)$ is a $\sigma$-closed cut, and $\bar{a}$, $\bar{a}'$ from $A$ and $\bar{b}$, $\bar{b}'$ from $B$ are such that $\qftp(\bar{a}) = \qftp(\bar{a}')$ and $\qftp(\bar{b}) = \qftp(\bar{b}')$, then $\qftp(\bar{a}\bar{b}) = \qftp(\bar{a}'\bar{b'})$ [``$\qftp$'' denotes the quantifier-free type].
\end{rmk}

The following theorem states a main obstruction to model completeness.

\begin{thm}
Let $T$ be an $\L_{<,\sigma}$-theory extending $\LO_\sigma$. If $T$ is model complete, then for every $\L_{<,\sigma}$-formula $\varphi(x)$, there is a number $k$ such that
$$T\models \mbox{``There are at most $k$ $\varphi$-split $\sigma$-closed cuts.''}$$
\end{thm}
\begin{proof}
Fix a formula $\varphi(x)$, and assume by way of contradiction that there is no such $k$. Then, by compactness, there is a model $\M\models T$ with $|T|^+$ $\varphi$-split $\sigma$-closed cuts, say $\langle C_i = (A_i, B_i)\;:\;i\in I\rangle$ with $|I| = |T|^+$. We may assume that $C_i << C_j$ (i.e., $A_i$ is a proper subset of $A_j$) whenever $i < j$. For each $i\in I$, choose $a_i^*\in A_i$ and $b_i^*\in B_i$ witnessing $\varphi$-splitting. Then $\M\models\theta(a_i^*, b_i^*)$, where
$$\theta(u, v) := \forall x, y\in(u, v)\Big[(\varphi(x)\wedge\neg\varphi(y))\rightarrow x < y\Big].$$
As $T$ is model complete,
$$T\cup\Delta_\M\vdash\theta(a_i^*, b_i^*)$$
for each $i$. By compactness, for each $i$, choose $\bar{a}_i$ from $A_i$ and $\bar{b}_i$ from $B_i$ and a quantifier-free formula $\delta_i(\bar{x}, \bar{y})$ such that $T+\delta_i(\bar{a}_i, \bar{b}_i)\vdash\theta(a_i^*, b_i^*)$. Without loss of generality, we may assume that $a_i^*$ and $b_i^*$ are the first elements of $\bar{a}_i$ and $\bar{b}_i$, respectively. By the pigeon-hole principle, we may assume that $\delta_i$ is constantly $\delta$ and that the quantifier-free type of $\bar{a}_i\bar{b}_i$ is constant for all $i$. But then, it follows from Remark 2.5 that $\bar{a}_i\bar{b}_{i+1}$ has the same quantifier-free type as $\bar{a}_i\bar{b}_i$. Hence, $\M\models\delta(\bar{a}_i, \bar{b}_{i+1})$, which implies that $\M\models\theta(a_i^*, b_{i+1}^*)$, which is a contradiction.
\end{proof}

We use this theorem to give a bound on the number of model complete theories extending $DLO_\sigma$ in the language $\L_{<, \sigma}$.

\begin{cor}
For any model complete $\L_{<, \sigma}$-theory $T$ extending $DLO_\sigma$, there are only finitely many split $\sigma$-closed cuts.
\end{cor}
\begin{proof}
The only interesting formulas in one-variable to be considered are $x < \sigma(x)$, $x = \sigma(x)$ and $\sigma(x) < x$. This can be achieved by the usual trick of renaming $\sigma^n(x)$, the $n$-th iterate of $\sigma$, by a new variable $y_n$ and using quantifier elimination of DLO in the language $\L_{<}$ and finally translating back any $y_n$ occurring in the resulting quantifier-free formula to $\sigma^n(x)$. The general quantifier-free $\L_{<, \sigma}$-formulas in one-variable are of the form $\sigma^m(x) < \sigma^n(x)$, but this is equivalent to $x < \sigma(x)$ if $m < n$, and to $\sigma(x) < x$ if $n < m$. It thus follows that $\varphi$-splitting by any formula $\varphi(x)$ corresponds to splitting by one of these three basic formulas. If $x = \sigma(x)$, then $x$ is the only element in its $\sigma$-archimedian class. If $x < \sigma(x)$ and $y$ belongs to the same $\sigma$-archimedian class as $x$, then there is some $m\in\Z$ such that $\sigma^m(x) \le y < \sigma^{m+1}(x)$. Consequently, we have $\sigma^{m+1}(x)\le\sigma(y) < \sigma^{m+2}(x)$, i.e., $y < \sigma(y)$. Similarly for the case $\sigma(x) < x$. In other words, these three formulas are preserved in a $\sigma$-archimedian class.

Thus, the only place where such a formula $\varphi(x)$ can possibly break is at a $\sigma$-closed cut. Let $T$ be a model complete $\L_{<, \sigma}$-theory extending $DLO_\sigma$. By Theorem 2.6, there is a number $k$ such that in every model $\M$ of the theory $T$, there are at most $k$ $\varphi$-split $\sigma$-closed cuts. Let $k_1$, $k_2$ and $k_3$ be the numbers associated with the formulas $\sigma(x) < x$, $\sigma(x) = x$ and $x < \sigma(x)$ respectively. Then there are at most $k_1 + k_2 + k_3$ split $\sigma$-closed cuts in $\M$. Thus, in any model complete $\L_{<, \sigma}$-theory $T$ extending $DLO_\sigma$, there are only finitely many split $\sigma$-closed cuts.
\end{proof}

We now use Theorem 2.6 and Corollary 2.7 to give a complete characterization of all model complete $\L_{<,\sigma}$-theories extending $DLO_\sigma$.

\begin{defn}
A model $\M$ has {\em length $n$} if it has exactly $n$ split $\sigma$-closed cuts.
\end{defn}

\begin{rmk}
If $\M$ has $n$ split $\sigma$-closed cuts, it is divided into $n+1$ blocks, where on each block one of three things happen: 
\begin{itemize}
\item[(C)] $\sigma(x) = x$ for all $x$, or
\item[(\,I\,)] $x < \sigma(x)$ for all $x$, or
\item[(D)] $\sigma(x) < x$ for all $x$. 
\end{itemize}
There are also 2 different ways in which blocks of type $C$ can occur, namely,
\begin{itemize}
\item[(C$^1$)] The block is a singleton;
\item[(C$^\emptyset$)] The block is a model of $DLO$ with no endpoints.
\end{itemize}
For a model $\M$ of length $n$, we then introduce the following notation:
\begin{eqnarray*}
C(M) & := & \{0\le k\le n\;|\; k^{th} \mbox{ block is of type C}\}\\
I(M) & := & \{0\le k\le n\;|\; k^{th} \mbox{ block is of type I}\}\\
D(M) & := & \{0\le k\le n\;|\; k^{th} \mbox{ block is of type D}\}
\end{eqnarray*}
\end{rmk}

We now show that the statements ``there are $n$ blocks'' and ``$x$ is in the $k^{th}$ block" are first-order.
\begin{lem}
The statements ``$x$ is in the $k^{th}$ block'' and ``there are $n$ blocks'' are first-order statements.
\end{lem}
\begin{proof}
Define an elementary equivalence relation $\approx$ on pairs $\{a, b\}$ from $\M$ as the reflexive and symmetric closure of the binary relation $R$, where 
\begin{eqnarray*}
R(a, b) \iff (a < b) & \bigwedge & \forall e\in(a, b)\;\Big[\Big(\sigma(a) = a\iff \sigma(e) = e\iff\sigma(b) = b\Big)\\
& & \bigvee\Big(a < \sigma(a)\iff e < \sigma(e)\iff b < \sigma(b)\Big)\\
& & \bigvee\Big(\sigma(a) < a\iff \sigma(e) < e\iff \sigma(b) < b\Big)\Big].
\end{eqnarray*}
The blocks are basically the $\approx$-classes, which are convex. Now
\begin{eqnarray*}
& & ``x \mbox{ is in the } k^{th} \mbox{ block''}\\
\iff & & \exists y_1 < y_2 < \ldots < y_k < x \bigwedge_{\begin{tabular}{c} $i, j = 1$\\$i\not= j$\end{tabular}}^k y_i\not\approx y_j\\
\bigwedge & & \forall y_1 < y_2 < \ldots < y_{k + 1} < x \bigvee_{\begin{tabular}{c} $i, j = 1$\\$i\not= j$\end{tabular}}^{k + 1} y_i\approx y_j.
\end{eqnarray*}
Also,
\begin{eqnarray*}
& & ``\mbox{there are $n$ blocks''}\\
\iff & & \forall x \bigvee_{k = 1}^n ``x \mbox{ is in the $k^{th}$ block''.} \qedhere
\end{eqnarray*}
\end{proof}

Given all the observations we made above, any model complete $\L_{<, \sigma}$-theory extending $DLO_\sigma$ must specify the following data: the number of $\approx$-classes, the alternation pattern among C, I and D, and for each C, one of the 2 choices.

\begin{thm}
Let $T$ be a first-order $\L_{<, \sigma}$-theory with the following axioms:
\begin{itemize}
\item $DLO\; + $ ``$\sigma$ is an $\L_<$-automorphism''
\item There are exactly $n$ $\approx$-classes for some fixed $n\in\omega$
\item The alternation pattern among C, I and D
\item For each block of type C, specify whether it is $C^1$ or $C^{\emptyset}$.
\end{itemize}
Then $T$ is complete, model complete, and weakly o-minimal. Moreover, these are all the complete and model complete $\L_{<, \sigma}$-theories extending $DLO_\sigma$.
\end{thm}
\begin{proof}
Let $\M\subseteq\mN\models T$. Since the number of $\approx$-classes, the alternation pattern among C, I and D, and also the specific type of C is specified, the only way $\mN$ extends $\M$ is by extending each block. But each block is a model of one of the model complete theories $\DLO$ (trivial or nontrivial), $\DLO_\sigma^+$ or $\DLO_\sigma^-$. Hence, $\M\preceq\mN$.

Moreover, since each of the above theories $\DLO$, $\DLO^+_\sigma$ and $\DLO^-_\sigma$ is a complete theory and there are only finitely many blocks, the theory $T$ is complete as well.
\end{proof}

As an immediate corollary, we get
\begin{cor}
There are exactly $\aleph_0$ complete and model complete $\L_{<, \sigma}$-theories extending $DLO_\sigma$.
\end{cor}

\begin{rmk}
Even though $\DLO$ is a model completion of $\LO$ and $\DLO^+_\sigma$ is a model completion of $\LO^+_\sigma$, it does not follow that every model complete extension of $\LO_\sigma$ is one of the theories listed in the theorem above. For example, consider the model complete theory $Th(\Z, <, \sigma)$, where $\sigma(n) = n+1$.
\end{rmk}

\section{Ordered Abelian Group with Increasing Automorphism}
Now we consider the theory $\OAG$ of ordered abelian groups in the language $\L_{OG} := \{+, -, 0, <\}$. As noted in the introduction, $\OAG$ has the strict order property and thus by Theorem 1.2, $\OAG_\sigma$ does not have a model companion in $\L_{OG, \sigma}$. So we need to put some restriction on the automorphism. A natural guess, inspired from the previous section and from the proof of the Kikyo-Shelah theorem, is to impose the restriction of an increasing automorphism. But in the context of groups, since $\sigma(-x) = -\sigma(x)$, we have $0 < x < \sigma(x)$ implies $\sigma(-x) < -x$. In particular, $\sigma$ cannot be increasing for all $x$. So we impose the following restriction:

\begin{defn}
An automorphism $\sigma$ of $\OAG$ is said to be \it{(positive) increasing} if
$$\forall x\;(x > 0\rightarrow x < \sigma(x)).$$
\end{defn}

We denote by $\ODAG$ the theory of ordered divisible abelian groups in the language $\L_{OG}$. We denote by $\ODAG^+_\sigma$ the theory of $\ODAG$ together with a (positive) increasing automorphism in the language $\L_{OG, \sigma}$. However, this restricted class of automorphisms does not work quite well because we will now show that $\ODAG^+_\sigma$ does not have a model companion in $\L_{OG, \sigma}$. The argument presented here is a variant of the proof of the Kikyo-Shelah theorem.

\begin{thm}
$\ODAG^+_\sigma$ does not have a model companion in $\L_{OG, \sigma}$.
\end{thm}
\begin{proof}
Suppose $\ODAG^+_\sigma$ has a model companion, say $T_A$. Recall that $\ODAG$ is a complete and model complete theory in $\L_{OG}$. 

Fix any $1 < r < q\in\Q$. Fix any $0 < a\in\Q$. It is easy to see that $\tp_{\L_{OG}}(a/\emptyset) = \tp_{\L_{OG}}(qa/\emptyset)$. By induction, it follows that $\tp_{\L_{OG}}(q^ia/\emptyset) = \tp_{\L_{OG}}(q^ja/\emptyset)$ for any non-negative integers $0 \le i, j < \omega$. Let $I = \langle a_i : i < \omega\rangle$, where $a_i = q^ia$. 

Define $\tilde\sigma: \Q\to\Q$ as $\tilde\sigma(x) = qx$. It is easy to check that $\tilde\sigma$ is an $\L_{OG}$-automorphism of $\Q$. Moreover, it is increasing too. Thus, $(\Q, \tilde\sigma)\models \ODAG^+_\sigma$. So we can embed $(\Q, \tilde\sigma)$ into a model $(N, \sigma)$ of $T_A$. In particular, $\sigma(a_i) = a_{i+1} = q^ia$. Without loss of generality, we can assume $(N, \sigma)$ is sufficiently saturated.

Now consider,
\begin{eqnarray*}
p(x) & = & \{a_i < x : i < \omega\} \\
\psi(x) & = & \exists y\;(a_0 < y < x \;\wedge\; \sigma(y) = ry)
\end{eqnarray*}

\textbf{Claim:} In $(N, \sigma)$,
\begin{eqnarray*}
(1) & p(x)\vdash\psi(x) & \\
(2) & q(x)\not\vdash\psi(x) & \mbox{for any finite subset } q(x) \subset p(x).
\end{eqnarray*}

We show (2) first. Suppose $q(x)\subset\{a_i < x : i < L\}$ for some $L < \omega$. Then, $a_L\in q(N)$. By way of contradiction, suppose that $a_L\in\psi(N)$. Let $y_0$ witness $\psi(a_L)$. Thus, we have
$$a_0 < y_0 < a_L\;\wedge\;\sigma(y_0) = ry_0.$$
In particular, since $\sigma$ is an $\L_{OG}$-automorphism, we have that, for all $k < \omega$,  
$$\sigma^k(a_0) < \sigma^k(y_0)\iff q^ka_0 < r^k y_0\iff \Big(\dfrac{q}{r}\Big)^ka_0 < y_0,$$
which implies $a_L < y_0$, since $a_L = q^La_0 < \Big(\dfrac{q}{r}\Big)^{k_0}a_0$ for some sufficiently large $k_0 < \omega$, as $\dfrac{q}{r} > 1$. This leads to a contradiction.

Now we prove (1). Suppose $c\in p(N)$. Then, $c > a_i = q^ia_0$ for all $i < \omega$. Let $(M, \sigma|_M)$ be a countable $\L_{OG, \sigma}$-elementary substructure of $(N, \sigma)$ such that $a_0, c\in M$. Now consider the following type $q(x)$ over $M$:
\begin{eqnarray*}
q(x) & = & \{m < x : m\in M \mbox{ and } m < a_i \mbox{ for some } i < \omega\}\\
&& \bigcup\;  \{x < m : m\in M \mbox{ and } a_i < m \mbox{ for all } i < \omega\}\\
& =: & M_1 \cup M_2
\end{eqnarray*}
It is easy to see that any finite subset of $q(x)$ is realized in $M$ by some $a_i$. Hence, $q(x)$ is finitely consistent and thus consistent. Let $d\in N$ realize $q(x)$. Clearly $d\in N\setminus M$. Consider the divisible hull $G$ of the ordered abelian group generated by $M$ and $d$. As a group, $G$ can be identified with $M \oplus \Q d$; as an ordered group $G$ can be identified with $M_1\oplus \Q d\oplus M_2$ with the reverse lexicographic ordering (see Section 4.2 for more details). Thus, $G\models \ODAG$ and $M\subsetneq G\subsetneq N$. Since $\ODAG$ is model-complete, we in fact have that
$$M \preccurlyeq_{\L_{OG}} G\preccurlyeq_{\L_{OG}} N.$$

Define $\tau: G\to G$ as $\tau(m_1\oplus sd\oplus m_2) = \sigma(m_1) \oplus rsd\oplus \sigma(m_2)$ (for all $s\in\Q$). It is easy to see that $\tau$ is an automorphism of $G$, $\tau|_M = \sigma|_M$, and $\tau(d) = rd$. Moreover, since $\sigma$ is increasing and $r > 1$, it follows that $\tau$ is also an increasing automorphism. In particular, $(G, \tau)\models \ODAG^+_\sigma$. Also,
\begin{eqnarray*}
&&(G, \tau)\models a_0 < d < c \;\wedge\; \sigma(d) = rd\\
& i.e., &(G, \tau)\models \exists y\;[a_0 < y < c \;\wedge\; \sigma(y) = ry].
\end{eqnarray*}

Since $(M, \sigma|_M)$ is a model of $T_A$, it is an existentially closed model of $\ODAG^+_\sigma$. Moreover, $(M, \sigma|_M) \subseteq (G, \tau)$. Thus, the formula $(a_0 < y < c \;\wedge\; \sigma(y) = ry)$ has a solution in $(M, \sigma|_M)$. Hence, $(M, \sigma|_M)\models\psi(c )$, and therefore, $(N, \sigma)\models\psi( c)$.
\end{proof}

\section{Model Complete Theories of Ordered Abelian Groups with Automorphism}
Inspired from Section 2, we will now show that, if we restrict ourselves to very specific kinds of automorphisms, we do actually get model complete theories of ordered abelian groups with an automorphism. We will first deal with the case of multiplicative automorphisms, and then we will give two general constructions for building model complete theories from other model complete theories in the context of ordered abelian groups with an automorphism.

\subsection{Multiplicative Automorphism}
In this section, each of the intended automorphisms is multiplication by an element of a real-closed field. For example, $\sigma(x) = 2x$, or $\sigma(x) = \sqrt{2}x$, or $\sigma(x) = \delta x$, where $\delta$ is an infinite or infinitesimal element.

The problem is that in general abelian groups such multiplications do not make sense. But since integers embed in any torsion-free abelian group, in particular any ordered abelian group, by imitating what we do for real numbers, we can make sense of such multiplication.
\newline
For an abelian group $G$, multiplication by $m\in\N$ makes sense: $mg := \overbrace{g + \cdots + g}^{m\mbox{ times}}.$ 
\newline
Taking additive inverses, multiplication by integers also makes sense: $(-m)g := -(mg).$
\newline
If $G$ is torsion-free divisible, then multiplication by rational numbers makes sense: $\dfrac{m}{n}g = \dfrac{mg}{n}$ is defined to be the unique $y\in G$ such that $ny = mg.$

\begin{motiv}
We carry this idea forward and define cuts in rational numbers to make sense of multiplication by irrationals. Let $\rho$ be an element of a real closed field $K$. Then, for any $0 < g\in G$, we would like $\rho\cdot g$ to be an element of $G$ such that, for all $r\in\Q$,
$$r g \lessgtr \rho\cdot g \iff r \lessgtr \rho.$$
Since we are typically interested in preserving the order on $G$, we also require that $\rho > 0$, because then $$g_1 \lessgtr g_2 \iff \rho\cdot g_1 \lessgtr \rho\cdot g_2.$$
Without loss of generality, we also require that $\rho\ge 1$; otherwise, we can work with $\rho^{-1}$ instead. 
Since $\rho$ is an element of the real closed field $K$, we can define the cut of $\rho$ in the rationals by
$$cut_\Q(\rho) = \{a\in K \;:\; \mbox{for each } q\in\Q,\; q \lessgtr a \iff q \lessgtr \rho\}.$$
Clearly for all $a\in cut_\Q(\rho)$, $\rho\cdot g$ and $a\cdot g$ are order-indistinguishable with respect to the rationals. This is a little bit of a problem because we would typically like to be able to distinguish between $b\cdot g$ and $(b + \epsilon)\cdot g$, where $b$ is an algebraic number, and $\epsilon$ is an infinitesimal. This is because if $b$ is algebraic over $\Z$, then $b$ is a root of a polynomial $L(x) = \sum_{i = 0}^n a_i x^i$, with $a_i\in\Z$ for all $i = 0,\ldots, n$. Then for any $0\not=g \in G$, we have $L(b)\cdot g = 0$, but $L(b + \epsilon)\cdot g\not= 0$. However, for any $a\in K$ and any polynomial $L(x)$ over $\Z$, we also have $L(a)\in K$. In particular, $L(a) > 0$ or $L(a) = 0$ or $L(a) < 0$. So either $L(a)\cdot g > 0$ for all $0 < g\in G$, or $L(a)\cdot g = 0$ for all $g > 0$, or $L(a)\cdot g < 0$ for all $g > 0$. This is the property we take away from this particular setting and apply to the general setting to make the ``multiplication'' work and define what we call multiplicative ordered difference abelian group ($\MODAG$).
\end{motiv}
\vspace{2mm}

Coming back to the general situation, we have an ordered abelian group $G$ and an automorphism $\sigma: G\to G$. For $i\in\N$, we denote
\begin{center}
$\sigma^{i}(x) := \overbrace{\sigma(\sigma(\ldots(\sigma}^{i\mbox{ times}}(x))\ldots))$. 
\end{center}

\begin{defn}
There is a natural map $\Phi: \Z[\sigma]\to End(G)$, which maps any $L := m_k \sigma^{k} + m_{k - 1}\sigma^{k - 1} + \ldots + m_1\sigma + m_0$ (thought of as an element of $\Z[\sigma]$ with the $m_i$'s coming from $\Z$), to an endomorphism $L(\cdot): G\to G$. Such an $L$ is called a {\em linear difference operator}.
\end{defn}

Due to this action of $\Z[\sigma]$, $G$ has the structure of a $\Z[\sigma]$-module, with the understanding that $\sigma$ has an inverse. To turn it into an ordered $\Z[\sigma]$-module, we further impose the following condition on $\sigma$ (motivated from our earlier example with the real closed fields): for each $L\in\Z[\sigma]$,
$$\Big(\forall x > 0\; (L(x) > 0)\Big)\bigvee\Big(\forall x > 0\; (L(x) = 0)\Big)\bigvee\Big(\forall x> 0\; (L(x) < 0)\Big).$$
We call this condition Axiom OM (OM stands for {\it Ordered Module}). This axiom also makes sense for $\sigma$ an injective endomorphism.

Axiom OM is consistent with the axioms of $\OAG$ because any ordered abelian group is a model of this axiom with $\sigma(x) = 2x$ for all $x$, say. Also, with this axiom, $\Z[\sigma]$ becomes a quasi-ordered ring with the order defined as follows: 
\begin{eqnarray*}
L_1 \geqq L_2 & \iff & \forall x > 0\; \Big((L_1 - L_2)(x) \geqq 0\Big),\mbox{ and }\\
L_1 > L_2 & \iff & \forall x > 0\; \Big((L_1 - L_2)(x) > 0\Big).
\end{eqnarray*}

It is easy to see that the relation 
$$L_1 \approx L_2 \iff L_1 \geqq L_2 \mbox{ and } L_2\geqq L_1\iff \forall x > 0\; ((L_1 - L_2)(x) = 0)$$ 
is an equivalence relation. Thus taking a quotient makes sense, and we define
\begin{defn}
$\Z[\rho] := \Z[\sigma]/\approx$, where $\rho$ is the image of $\sigma$ under this quotient map.

We also define $\Q(\rho)$ to be the fraction field of $\Z[\rho]$.
\end{defn}

\begin{rmk}
Clearly then $\Z[\rho]$ is an (totally) ordered ring and admits an embedding into a real closed field. So $\rho$ can also be simultaneously thought of as an element of a real closed field.

It is also easy to see that $\Z[\rho] = \Z[\sigma]/Ker(\Phi)$, where $\Phi$ is as defined in Definition 4.1. Note that the kernel of $\Phi$ need not be trivial. For example, if $\sigma(x) = 2x$ for all $x$, then $\sigma - 2\in $ Ker$(\Phi)$.

Moreover $G$ is an ordered module over the ordered ring $\Z[\rho]$ with the understanding that $\rho$ has an inverse. So we can denote the automorphism on $G$ equivalently by $\rho\cdot$, i.e. $\sigma(x) = \rho\cdot x$. Axiom OM then is equivalent to: for each $L\in\Z[\rho]$,
$$\Big(\forall x > 0\; (L\cdot x > 0)\Big)\bigvee\Big(\forall x > 0\; (L\cdot x = 0)\Big)\bigvee\Big(\forall x> 0\; (L\cdot x < 0)\Big).$$
\end{rmk}

\begin{defn}
For any ordered difference abelian group $G$ satisfying Axiom OM, we define the {\em set of $\Z[\sigma]$-positivities of $G$} as
$$ptp_{\Z[\sigma]}(G) := \{L\in\Z[\sigma] : \forall x \in G\; (x>0\rightarrow L(x) > 0)\}.$$
We say that $G$ and $G'$ have the {\em same $\rho$} if $ptp_{\Z[\sigma]}(G) = ptp_{\Z[\sigma]}(G')$. We also say $G$ is a {\em $\MODAG$ with a given $\rho$} if $G$ satisfies a given consistent set of $\Z[\sigma]$-positivities.
\end{defn}

\begin{defn}
An ordered difference abelian group is called \textit{multiplicative} (in short, $\MODAG$) if it satisfies Axiom OM. The theory of such structures (also called $\MODAG$) is axiomatized by the axioms of an ordered difference abelian group together with Axiom OM. Note that this is an $\forall\exists$-theory.

We denote by $\MODAG_\rho$ the theory of the class of all $\MODAG$s with a same $\rho$.
\end{defn}

\begin{defn}
If there is a non-zero $L\in\Z[\sigma]$ such that $\forall x > 0 (L(x) = 0)$, we say $\rho$ satisfies $L$ and $\rho$ is {\em algebraic} (over the integers); otherwise $\rho$ is {\em transcendental}. If $\rho$ is algebraic, there is a minimal (degree) polynomial that $\rho$ satisfies.
\end{defn}

\begin{defn}
A $\MODAG$ $G$ is called \textit{divisible} (in short, $\MODDAG$) if for any $0\not\approx L \in\Z[\sigma]$ and $b\in G$, the equation $L(x) = b$ has a solution in $G$.
\end{defn}

\begin{defn}
The $\L_{OG, \sigma}$-theory of nontrivial multiplicative ordered divisible difference abelian groups (also denoted by $\MODDAG$) is axiomatized by the axioms of a $\MODAG$ along with 
$$\exists x\; (x\not=0)$$
and the following additional infinite list of axioms: for each $L\in\Z[\sigma]$,
$$\Big(\forall x\; (L(x) = 0) \Big) \vee \Big(\forall y\exists x\; (L(x) = y)\Big),$$
i.e., all non-zero linear difference operators are surjective (and because of Axiom OM, it follows in this case that $\forall y\exists!x (L(x) = y)$). Thus, $\MODDAG$ is an $\forall\exists$-theory. Similarly as above, we denote by $\MODDAG_\rho$ the theory of the class of all $\MODDAG$s with a same $\rho$.
\end{defn}

\begin{rmk} 
It might already be clear from the definitions above that, for a given $\rho$, $\MODDAG_\rho$ is basically the theory of nontrivial ordered vector spaces over the ordered field $\Q(\rho)$. Quantifier elimination then follows from well-known results \cite{LD}. \end{rmk}

\begin{thm}
$\MODDAG_\rho$ is complete and has quantifier elimination in $\L_{OG, \sigma}$. Also, $\MODDAG_\rho$ is o-minimal. Moreover, $\MODDAG$ is the model companion of $\MODAG$ in $\L_{OG, \sigma}$.
\end{thm}
\begin{proof}
See \cite[Lemma 2.12, Lemma 2.14, Theorem 2.15]{KP}.
\end{proof}

\begin{rmk}
Observe that, because of Axiom OM, div-MODAGs are essentially of 4 different kinds up to elementary equivalence:
\begin{itemize}
\item ``Non-algebraic type'' : Indexed by a transcendental real number $\rho\in(0,\infty]$
\item ``Algebraic type with equality'' : Indexed by a real algebraic number $\rho\in(0, \infty)$ such that $\sigma(x) = \rho\cdot x$ for all $x > 0$, and is denoted by ``=$\rho$''
\item ``Algebraic type with $<$'':  Indexed by a real algebraic number $\rho\in(0, \infty)$ such that $\sigma(x) < \rho\cdot x$ but infinitesimally close to $\rho\cdot x$ for all $x > 0$, and is denoted by ``$<\!\!\rho$''
\item ``Algebraic type with $>$'':  Indexed by a real algebraic number $\rho\,\in(0, \infty)$ such that $\sigma(x) > \rho\cdot x$ but infinitesimally close to $\rho\cdot x$ for all $x > 0$, and is denoted by ``$>\!\!\rho$''.
\end{itemize}
\end{rmk}

\subsection{$n$--sums}
In the previous subsection, we saw that a uniform increasing behavior of $\sigma$ is good enough to get a model companion. But we do not need it to be uniformly increasing on the whole universe -- that is too restrictive. It suffices to have $\sigma$ behave uniformly on pieces. We make sense of this in our first construction -- the direct sum construction. We postpone the discussion of our second construction, which is a quotient construction, to the appendix.

\begin{defn}
For $n\ge 1$, $G=\bigoplus_{i<n}H_i$ is called an {\it $n$--sum} if each $H_i$ is a div-MODAG. The operations $+$, $-$, $\sigma$ are all defined coordinatewise, and the linear ordering is `reverse lexicographic', i.e., $g = (g_0, \ldots, g_{n - 1}) < 0$ iff $H_i\models g_i < 0$, where $i$ is maximal such that $g_i\not= 0$.
\end{defn}

Our goal is to prove the following theorem.
\begin{thm}
For all $n\ge 1$, every $n$--sum $G$ has a model complete $\L_{OG, \sigma}$-theory.
\end{thm}
To prove this theorem we need to first understand $n$--sums as structures in various other languages as follows. Let
\begin{itemize}
\item $\tau := \L_{OG, \sigma} = \{+, -, 0, <, \sigma\}$
\item $\tau_n := \tau \cup \{U_i \mid i \le n\}$
\item $\tau_n^* := \tau_n \cup \{K_L, R_L, k_L, r_L\mid L \mbox{ is a unary linear difference operator}\}$,
\end{itemize}
where each $U_i$, $K_L$ and $R_L$ are unary predicate symbols, and each $k_L$ and $r_L$ are unary function symbols. Although the above theorem sounds very expected, our proof technique is rather unique. We expand the language $\tau$ by adding predicates which are not in general definable in $\tau$. Our approach is to first show that $Th(G)$ eliminates quantifiers in the bigger language $\tau_n^*$, which immediately implies that $Th(G)$ is model complete in the language $\tau_n$. We then show that this implies that $Th(G)$ is model complete in the language $\tau$. It is noteworthy, however, that an arbitrary $n$-sum may have some subsets that are $\tau_n$-definable, but not $\tau$-definable. For example, even if $G$ is reduced (see Definition 4.21), $K_L\cap U_\ell$ might not be $\tau$-definable for some choice of $L$ and $\ell$. We start with the following definition.

\begin{defn}
A {\it graded $n$--sum}, considered as a $\tau_n$-structure, is an $n$--sum expanded by interpreting each $U_i$ as $\bigoplus_{j < i}H_j$ (for $i = 0$, we interpret $U_0$ by $\{0\}$).
\end{defn}

Since each ``coordinate'' $H_i$ of an $n$--sum $G$ is a div-MODAG, every linear difference operator $L$ on $H_i$ is either trivial or is $1-1$. If $G$ shares a similar property, understanding the model theory of $G$ becomes very easy and we call such a $G$ dull.

\begin{defn}
A graded $n$-sum $G$ is {\it dull} if, for every linear difference operator $L$, we have $\ker_G(L) = G$ or $\ker_G(L) = \{0\}$.
\end{defn}

\begin{thm}
Let $G$ be any dull graded $n$--sum. Then $Th(G)$ eliminates quantifiers in the language $\tau_n$.
\end{thm}
\begin{proof}
The proof is just like the proof for div-MODAGs.
\end{proof}

So the interesting $n$--sums are those which are not dull. To understand these, we need to fix a few notations first. These notations in fact make sense for any arbitrary $n$--sum. So fix an $n$--sum $G$ and a unary linear difference operator $L: G\to G$. Let $K_L := \ker_G(L)$ and $R_L := \mbox{range}_G(L)$. Clearly these sets are $\tau$-definable. $K_L$ is quantifier-free definable, while $R_L$ appears to require an existential formula. It is also easy to see that
$$G\models \forall g\exists!a\exists!b\exists c(g = a + b \;\wedge\; L(a) = 0\;\wedge\; L(c) = b).$$
So, the group $G$ is abstractly isomorphic to $K_L+R_L$. Moreover, this isomorphism is $\tau$-definable: Choose $g\in G$. Write $g = a+b$ with $a\in K_L$ and $b\in R_L$. Then, $A_g := \{h\in G\mid L(h - g) = 0\}$ is $\tau$-definable and is a coset of $K_L$. Thus, $b$ is the unique element of $A_g\cap R_L$, and $a = g - b$. Let $k_L:G\to K_L$ and $r_L:G\to R_L$ denote the maps $g\mapsto a$ and $g\mapsto b$ respectively. Thus, every graded $n$--sum has a natural, 0-definable expansion to a $\tau^*_n$-structure.

Each of these functions $k_L$ and $r_L$ are essentially projections. Specifically, for any fixed $L$, let
$$s := \{i< n\mid \ker_{H_i}(L) = H_i\}\;\;\;\;\;\;\mbox{ and }\;\;\;\;\;\; G_s := \{g\in G\mid g_i = 0 \mbox{ for all } i\not\in s\}.$$
Then $K_L = G_s$ and $k_L$ is the `projection map' $\pi_s: G\to G_s$. Similarly, $R_L = G_{n\setminus s}$ and $r_L$ is the map $\pi_{n\setminus s}$. For each graded $n$--sum $G$, there is a boolean algebra $\mathcal{F}\subseteq\mathcal{P}(n)$ of subsets of $n$, such that for every $s, t\in\mathcal{F}$ with $s\subseteq t$, we have a definable map $\pi_s:G_t\to G_s$. It follows immediately that all of the $\tau^*_n$-operations -- addition, subtraction, $\sigma$, and all of the $\pi_s$'s -- commute with each other.

If $G$ is dull, then for every $L$, $K_L$ and $R_L$ are either $\{0\}$ or $G$, and the functions $k_L, r_L$ are either the identity or constantly zero. Moreover, the associated boolean algebra $\mathcal{F}$ consists of only two elements $\{\emptyset, n\}$.

Now suppose that a given $n$--sum $G = \bigoplus_{i<n}H_i$ is not dull. Fix a particular $L^*$ witnessing this, i.e., $\ker_G(L^*)$ is neither $G$ nor $\{0\}$. To simplify notation, let us denote $K_{L^*}$ by $K$, $R_{L^*}$ by $R$, let $s = \{i\in n\mid \ker_{H_i}(L^*) = H_i\}$ and let $t = n\setminus s$. Define the enumeration functions $i: |s| \to n$ and $j: |t|\to n$ by $i(0) = j(0) = 0$, $i(\alpha + 1)$ be the least element of $s$ that is greater than $i(\alpha)$, and $j(\beta + 1)$ be the least element of $t$ that is greater than $j(\beta)$.

\begin{defn}
A {\it $K$-formula} is a $\tau_n^*$-formula $\varphi$ such that the only $U_\alpha$'s that appear in $\varphi$ are such that $\alpha\in s$. Dually, an {\it $R$-formula} has only $U_\beta$'s with $\beta\in t$.
\end{defn}

We will prove by induction that for every graded $n$--sum $G$, $Th(G)$ admits elimination of quantifiers in the language $\tau_n^*$. The following proposition is the key to our induction.
\begin{prop}
Let $G$ be any non-dull graded $n$--sum. Then every quantifier-free $\tau_n^*$-formula $\theta(\bar{z})$ is $Th(G)$-equivalent to a boolean combination of $K$-formulas $\theta_i^K(k(\bar{z}))$ and $R$-formulas $\theta_j^R(r(\bar{z}))$.
\end{prop}
\begin{proof}
As both $K$-formulas and $R$-formulas are closed under negation, it suffices to prove that every atomic $\tau_n^*$-formula is of this form. Additionally, since the unary functions $k(x)$ and $r(x)$ commute with every term, it suffices to show that the atomic $\tau_n^*$-formulas $U_\alpha(x)$, $K_L(x)$, $R_L(x)$, and $x < y$ each have such a representation.

To see this, we start with each $U_\alpha(x)$. As $U_0 = \{0\}$, it is equivalent to both a $K$-formula and an $R$-formula. So assume $\alpha > 0$. There are two cases, depending on whether or not $\alpha\in s$. First, suppose that $\alpha\in s$. Then $U_\alpha(x)$ is a $K$-formula. Choose $\delta$ to be the maximal element of range$(j)$ that is below $\alpha$. Then, for any $g\in G$, $U_\alpha(g)$ iff $g_m = 0$ for all $m\in[\alpha, n)$ iff $U_\alpha(k(g))\wedge U_\delta(r(g))$. On the other hand, if $\alpha\in n\setminus s$, then $U_\alpha(x)$ is an $R$-formula. So choose $\gamma\in\mbox{range}(i)$ to be the maximal element that is below $\alpha$. In this case, $U_\alpha(x)$ is equivalent to $U_\gamma(k(x))\wedge U_\alpha(r(x))$.

The verifications that each $K_L, R_L$ (for various $L$) is equivalent to the required form is similar. Finally, in order to show $x < y$ is of this form, it suffices to show that `$x < 0$' has this form (since $x < y$ iff $y - x < 0$). To see this, for $1\le \alpha < n$, let $\delta_\alpha(x) := `k(x) < 0$' if $\alpha\in s$, and $\delta_\alpha(x) := `r(x) < 0$' if $\alpha\not\in s$. It can then be readily checked that `$x < 0$' is $Th(G)$-equivalent to
$$\bigvee_{1\le \alpha \le n} \big(U_\alpha(x) \wedge \neg U_{\alpha - 1}(x) \wedge \delta_\alpha(x)\big),$$
which is a boolean combination of $K$- and $R$-formulas from above.
\end{proof}

\begin{cor}
For every $n\ge 1$, for every graded $n$--sum $G$, $Th(G)$ admits elimination of quantifiers in the vocabulary $\tau_n^*$.
\end{cor}
\begin{proof}
We prove this by induction on $n\ge 1$. For $n = 1$, $G$ is necessarily dull, so this follows from Theorem 4.16. So assume that the corollary holds for all $n'<n$ and fix a graded $n$--sum $G$. If $G$ is dull, then again the result follows from Theorem 4.16. So assume that $G$ is not dull. Fix an $L^*$ witnessing this as above, and use the notation from there.

It suffices to show that for any quantifier-free $\tau_n^*$-formula $\theta(x, \bar{y})$, $\exists x\theta(x, \bar{y})$ is $Th(G)$-equivalent to a quantifier-free formula $\psi(\bar{y})$.

Fix such a $\theta(x, \bar{y})$. By Proposition 4.18 and the fact that $K$-formulas and $R$-formulas are both closed under conjunction, $\theta(x, \bar{y})$ is equivalent to a disjunction of a conjunction of a $K$-formula and an $R$-formula. As existential quantification commutes with disjunction, it suffices to show that
$$\exists x\big[\theta^K(k(x), k(\bar{y})) \wedge \theta^R(r(x), r(\bar{y}))\big]$$
is $Th(G)$-equivalent to a quantifier-free formula. But, by our decomposition result, the displayed equation is equivalent to
$$(\exists x\in K)\theta^K(x, k(\bar{y})) \wedge (\exists x\in R)\theta^R(x, r(\bar{y})).$$
Observe that $K$ and $R$ can be viewed as being ``isomorphic'' to $|s|$--sum and $|t|$--sum respectively. By our choice of $L^*$, both $|s|$ and $|t|$ are strictly less than n. So induction hypothesis applies and both of the formulas above are $Th(K)$--equivalent and $Th(R)$--equivalent to quantifier-free formulas.
\end{proof}

\begin{rmk}
Since $\tau_n^*$ is a 0-definable expansion of $\tau_n$ and that too by at most existential formulas, it follows immediately from Corollary 4.19 that every graded $n$--sum $G$ has a model complete $\tau_n$-theory (as opposed to $\tau_n^*$).
\end{rmk}

To finish the proof of Theorem 4.13, we need to establish model completeness in the smaller language $\tau$. This is not immediate, as a graded $n$--sum may have more definable sets than its $\tau$-reduct. To that end, we make the following definition.
\begin{defn}
A (graded) $n$--sum $G = \bigoplus_{i < n}H_i$ is {\it reduced} if $H_i\not\equiv H_{i+1}$ for all $i < n-1$.
\end{defn}

\begin{lem}
Let $G$ be a reduced $n$--sum. Then each $U_\alpha$, $\alpha \le n$, is definable by a universal and an existential $\tau$-formula.
\end{lem}

To prove this result, we need another definition. Observe that if $G = \bigoplus_{i < n}H_i$ is an $n$--sum, then each $H_i$, being a div-MODAG, is one of the 4 species mentioned in Remark 4.11 and has a real number $\rho_i$ associated to it.

Fix rational numbers $q_i < r_i\in\Q\cup\{+\infty\}$, for $i < n$, such that
\begin{itemize}
\item $\rho_i\in(q_i, r_i)$
\item for $i\not=j$, the intervals $(q_i, r_i)$ and $(q_j, r_j)$ are either the same or disjoint.
\end{itemize}
(If some $H_i$ is of ``infinite type'', choose $r_i = +\infty$.)

\begin{defn}
A sequence $\langle a_i \mid i < n \rangle$ of elements from $G$ is a {\it representative sequence} if
\begin{enumerate}
\item $0 < a_0 < \cdots < a_{n - 1}$
\item $q_i a_i < \sigma(a_i) < r_i a_i$, for each $i < n$
\item If $H_i$ is of type ``=$\rho_i$'' for some algebraic number $\rho_i$, then $L(a_i) = 0$, where $L$ is a minimal polynomial that $\rho_i$ satisfies, for each $i < n$
\item If $H_i$ is of type ``$<\!\!\rho_i$'' for some algebraic number $\rho_i$, then $L(a_i) < 0$, where $L$ is a minimal polynomial that $\rho_i$ satisfies, for each $i < n$
\item If $H_i$ is of type ``$>\!\!\rho_i$'' for some algebraic number $\rho_i$, then $L(a_i) > 0$, where $L$ is a minimal polynomial that $\rho_i$ satisfies, for each $i < n$
\end{enumerate}
\end{defn}

The following obvious result is the key lemma about representative sequences.

\begin{lem}
Suppose $G$ is a reduced $n$--sum. Then
\begin{enumerate}
\item Representative sequences exist
\item If $\langle a_i\mid i < n\rangle$ is any representative sequence, then $a_i\in U_{i+1}\setminus U_i$ for $i < n$ (recall $U_0 = \{0\}$ and $U_n = G$)
\item The formula $\theta(x_0, \ldots, x_{n-1})$ asserting that ``$\langle x_0, \ldots, x_{n-1}\rangle$ is a representative sequence'' is quantifier-free definable in $\tau$.
\end{enumerate}
\end{lem}

We are now ready to prove Lemma 4.22.
\begin{proof}
For $\alpha = 0$, since $U_0 = \{0\}$, $U_0(x)$ holds for any $x\in G$ iff $x = 0$, which is a quantifier-free $\tau$-formula, and hence both existential and universal as well.

For $\alpha = n$, since $U_n = G$, $U_n(x)$ holds for any $x\in G$ iff $x = x$, which is again a quantifier-free $\tau$-formula, and hence both existential and universal as well.

So let $0 < \alpha < n$ and let $x\in G$ with $x > 0$. Then,
\begin{eqnarray*}
U_\alpha(x) \mbox{ holds } & \mbox{ iff } & \exists \mbox{ a representative sequence } \langle a_0, \ldots, a_{n-1}\rangle\mbox{ s.t. } x < a_{\alpha-1}\\
& \mbox{ iff } & \forall \mbox{ representative sequences } \langle a_0, \ldots, a_{n-1}\rangle\mbox{ we have } x < a_{\alpha}.
\end{eqnarray*}
Thus, each $U_\alpha$, $\alpha \le n$, is definable by a universal and an existential $\tau$-formula.
\end{proof}

Combining Remark 4.20 and Lemma 4.22 yields that if $G$ is a reduced $n$--sum, then $Th(G)$ is model complete in the language $\tau$. It is an easy observation that if $G = \bigoplus_{i < n}H_i$ is an $n$--sum with $H_j\equiv H_{j+1}$ for some $0 \le j < n - 1$, then $G$ is $\tau$-elementarily equivalent to $G'$, which is the $(n - 1)$--sum formed from $G$ by eliminating $H_{j+1}$. In particular, every $n$--sum is $\tau$-elementarily equivalent to some reduced $m$--sum. This completes the proof of Theorem 4.13.

As a corollary to this theorem, we get more examples of model complete groups with an automorphism. For example, the ordered abelian group $\Q\oplus\Q$, with automorphism $\sigma$ defined as $\sigma(a\oplus b) = 2a\oplus 3b$, is model complete in $\L_{OG, \sigma}$.

In the following subsection, we embark on an attempt to classify when an $\omega$-direct sum of div-MODAGs is model complete in $\L_{OG, \sigma}$.

\subsection{$\omega$--sums.} We begin this subsection with the following two definitions.
\begin{defn}
An {\em $\omega$-sum} is an $\L_{OG, \sigma}$-structure $G := \oplus_{i\in\omega}H_i$, where each $H_i$ is a div-MODAG. Note that this is a sum as opposed to a product, which means every element has finite support. The operations of $+, -, \sigma$ are all defined coordinatewise, and the linear ordering is reverse lexicographic, i.e., $g < 0$ if and only if $H_k\models g_k < 0$, where $g = \oplus_{i\in\omega}g_i$ and $k$ is maximal such that $g_k\not= 0$.

A {\em graded $\omega$-sum}, considered as a $\tau_\omega$-structure where $\tau_\omega := \L_{OG, \sigma}\cup\{U_i\mid i\in\omega\}$, is an $\omega$-sum expanded by interpreting each unary predicate symbol $U_i$ as $\oplus_{j< i}H_j$ (we interpret $U_0$ by $\{0\}$).
\end{defn}

\begin{defn}
Let $\M$ be a $\tau_\omega$-elementary extension of a graded $\omega$-sum $G$. An element $a\in\M$ is called {\em standard} if there is some $\ell\in\omega$ such that $\M\models U_\ell(a)$; otherwise $a$ is called {\em nonstandard}. We also define the {\em nonstandard part of $M$} as the collection of all nonstandard elements of $\M$, i.e., 
$$nonstandard(\M) = \M\setminus\{a\in\M\mid \M\models U_\ell(a) \mbox{ for some } \ell\in\omega\}.$$
\end{defn}

\begin{rmk}
Let $G = \oplus_{i\in\omega}H_i$ be a graded $\omega$-sum and $\M$ a $\tau_\omega$-elementary extension of $G$. Then for any $n > 0$, $U_n(G)$ is clearly a substructure of $U_n(\M)$. Since both $U_n(G)$ and $U_n(\M)$ are models of the theory of an $n$-sum, and the theory of an $n$-sum is model complete in the language $\tau_\omega$, it follows that $U_n(G)$ is in fact an elementary substructure of $U_n(\M)$ in the language $\tau_\omega$ (and hence in the language $\L_{OG, \sigma}$). In other words, the standard part of $\M$ behaves similarly as $G$. It is only the nonstandard part of $\M$ that is more interesting.
\end{rmk}

The main goal of this subsection is to show that if the theory of an $\omega$-sum is model complete, then it must have a unique type at infinity. In fact, the theory of a graded $\omega$-sum (and hence, the theory of an $\omega$-sum) is determined by its restriction to each of the finite $n$-sums along with its behavior at infinity. And if there is not a unique behavior at infinity, the existence of a model companion is ruled out. Our main theorem is the following, where we give a complete characterization of which $\omega$-sums have model complete first-order theories.
\begin{thm}
Let $G = \oplus_{i\in\omega}H_i$ be an $\omega$-sum, considered as a structure in the language $\L_{OG, \sigma}$. For each algebraic $\rho$, define
\begin{itemize}
\item[(i)] $Con^G_\rho := \{i\in\omega\mid H_i\models\forall x(\sigma(x) = \rho\cdot x)\}$
\item[(ii)] $Inc^G_\rho := \{i\in\omega\mid H_i\models\forall x>0(\sigma(x) > \rho\cdot x)\}$
\item[(iii)] $Dec^G_\rho := \{i\in\omega\mid H_i\models\forall x>0(\sigma(x) < \rho\cdot x)\}$
\end{itemize}
Then, $Th(G)$ is model complete if and only if for all algebraic $\rho$, one of the sets $Con^G_\rho$, $Inc^G_\rho$ or $Dec^G_\rho$ is cofinite.
\end{thm}

Thus, for example, $Th(\oplus_{i\in\omega}H_i)$, where $H_n$ is a model of $\MODDAG_{\rho_n}$ such that $\rho_n\cdot x = (n+1)x$ for each $n\ge 0$, is model complete in $\L_{OG, \sigma}$. On the other hand, the theory of $\Q_2\oplus\Q_3\oplus\Q_2\oplus\Q_3\oplus\cdots$(repeated $\omega$ times), where $\Q_n$ is $(\Q,+,-,<,\sigma_n)$ with $\sigma_n(x) = nx$, is not model complete. We will prove the theorem through a series of results and definitions, but at its core, the reason why certain $\omega$-sums are not model complete is akin to the non-model completeness of $Th(\omega,<,0)$. That is, the immediate successor relation is not existentially definable. 

\begin{defn}
An {\em unpacked atomic formula} in the language $\L_{OG, \sigma}$ has the form
\begin{tabular}{lllllll}
$x_i + x_j = x_k$, & $x_i - x_j = x_k$, & $\sigma(x_i) = x_j$, & $x_i = x_j$, & $x_i < x_j$, & $x_i = 0$.
\end{tabular}
\end{defn}

\begin{rmk}
It is easy to see that every atomic formula in $\L_{OG, \sigma}$ can be written as a conjunction of unpacked atomic formulas in $\L_{OG, \sigma}$ by introducing more variables. Since an atomic formula in $\L_{OG, \sigma}$ is a linear difference equation $L(\bar{x}) = 0$ or inequation $L(\bar{x}) < 0$ (or $L(\bar{x}) > 0$) in many variables, it is enough to consider only such formulas. For example, an unpacked form of $\sigma^2(x_1) - 2\sigma(x_2) + 3x_3 = 0$ is
\begin{eqnarray*}
(x_4 = \sigma(x_1))\wedge(x_5 = \sigma(x_4))\wedge(x_6 = \sigma(x_2))\wedge(x_7 = x_6 + x_6)\wedge(x_8 = x_3 + x_3)\wedge\\
(x_9 = x_8 + x_3)\wedge(x_{10} = x_5 - x_7)\wedge(x_{11} = x_{10} + x_9)\wedge(x_{11} = 0).
\end{eqnarray*}
More precisely, for any atomic formula $\alpha(\bar{x})$ in $\L_{OG, \sigma}$, there is a conjunction $\beta(\bar{x},\bar{z})$ of unpacked atomic formulas in $\L_{OG, \sigma}$ such that
$$\vdash\forall\bar{x}[\alpha(\bar{x})\leftrightarrow\exists\bar{z}\beta(\bar{x},\bar{z})].$$
It follows that any quantifier-free $\L_{OG,\sigma}$-formula is equivalent to a (positive) boolean combination of unpacked atomic formulas in $\L_{OG, \sigma}$ in possibly more variables. 
\end{rmk}

Now we introduce a notation that is needed to prove the lemma that follows.
\begin{notation}
Fix an $\omega$-sum $G$ and $\ell\in\omega$.\newline
For an element $b$ in $G$, let $b^-$ and $b^+$ be the elements of $G$ satisfying respectively
\begin{eqnarray*}
(b^-)_i & = & \left\{\begin{tabular}{ll}$(b)_i$ & if $i<\ell$\\0 & if $i\ge\ell$\end{tabular}\right.\\
(b^+)_i & = & \left\{\begin{tabular}{ll}0 & if $i<\ell$\\$(b)_i$ & if $i\ge\ell$\end{tabular}\right.
\end{eqnarray*}
For elements $b$ and $c$ in $G$, let $b^-c^+$ be the hybrid
$$b^-c^+ := b^- + c^+ = \left\{\begin{tabular}{ll}$b_i$ & if $i<\ell$\\$c_i$ & if $i\ge\ell$\end{tabular}\right.$$
For tuples $\bar{b} = (b_1, \ldots, b_n)$ and $\bar{c} = (c_1, \ldots, c_n)$ from $G$ of the same length, let $\overline{b^-c^+}$ denote the hybrid tuple $\overline(b_1^-c_1^+, \ldots, b_n^-c_n^+)$; and let $\dfrac{1}{2}\bar{b}$ denote $(\dfrac{1}{2}b_1, \ldots, \dfrac{1}{2}b_n)$.\newline
Also note that if $c> U_\ell$, then $b^-c^+ > \dfrac{1}{2}c$ for any $b\in G$.\newline
For a tuple $\bar{y} = (y_1, \ldots, y_n)$, we write $\bar{y} > x$ and $\bar{y} > U_\ell$ to respectively mean $y_i > x$ and $y_i > U_\ell$ for each $i = 1, \ldots, n$.
\end{notation}

\begin{defn}
Let $\M$ be an elementary extension of an $\omega$-sum $G$, and let $M_{\ge 0} :=\{m\in M\mid m\ge 0\}$. A cut $C = (A, B)$ of $\M_{\ge 0}$ is said to be {\em strong} if there is a different rate of growth across the cut, i.e., there is some $\rho$ such that 
\begin{eqnarray*}
& \forall a_1\in A\;\exists a_2\in A(a_1 < a_2 \;\wedge\; \sigma(a_2) \le \rho\cdot a_2) & \mbox{ and }\\
& \forall b_1\in B\;\exists b_2\in B(b_2 < b_1 \;\wedge\; \sigma(b_2) > \rho\cdot b_2) &
\end{eqnarray*}
or vice versa. In particular, if $G$ is a reduced $\omega$-sum and $C = (A, B)$ is a strong cut of $G_{\ge 0}$, then $A = U_\ell\cap G_{\ge 0}$ for some $\ell\in\omega$.
\end{defn}

Suppose $0 < a^*$ is some distinguished element of an $\omega$-sum $G$, and the formula $\varphi(x, a^*)$ defines a strong cut of $G_{\ge 0}$. Then, for a conjunction $\gamma(\bar{x},\bar{y})$ of unpacked atomic formulas, call a tuple $\bar{a}\subseteq G$ containing $a^*$ {\em good for $\gamma$ (with respect to $\varphi$}) if:
\begin{itemize}
\item[1.] $\bar{a}\subseteq\varphi(G,a^*)$
\item[2.] $(\exists\bar{b} > \varphi(G, a^*))\gamma(\bar{a},\bar{b})$.
\end{itemize}
This is clearly a first-order condition on $\bar{a}$. We then have the following result.

\begin{lem}
Fix $G$, $a^*$, $\varphi(x, a^*)$ and $\gamma(\bar{x},\bar{y})$ as above. Then there is a formula $\gamma^*(\bar{w})$, with length$(\bar{w}) = $ length$(\bar{y})$, such that $G$ satisfies the following two sentences: 
\begin{itemize}
\item[(a)] $(\forall\bar{a} \mbox{ good for }\gamma)(\forall\bar{b} > \varphi(G, a^*))$
$[\gamma(\bar{a},\bar{b}) \to\exists \bar{w}((\bar{w} > (1/2)\bar{b}) \wedge\gamma^*(\bar{w}))]$, and
\item[(b)] $(\forall \bar{a} \mbox{ good for }\gamma)(\forall\bar{w} > \varphi(G, a^*))$
$[\gamma^*(\bar{w}) \to \exists\bar{b}((\bar{b} > (1/2)\bar{w}) \wedge\gamma(\bar{a}, \bar{b}))]$
\end{itemize}
\end{lem}
\begin{proof}
Let $C = (A, B)$ be the cut defined by $\varphi(x, a^*)$. Since this is a strong cut, $A = U_\ell\cap G_{\ge 0}$ for some $\ell\in\omega$.

We first deal with the case when $\gamma(\bar{x}, \bar{y})$ is a single unpacked formula. The proof involves looking carefully at the sequence of coordinates and figuring out which relations can hold between $\bar{a}$ and $\bar{b}$ if it is given that $0 \le \bar{a}\subseteq U_\ell$ and $\bar{b}>U_\ell$. For example, it cannot happen that $a_i + a_j = b_k$, $a_i - a_j = b_k$, $b_i + b_j = a_k$, $\sigma(b_k) = a_i$, $\sigma(a_i) = b_k$, $b_k = 0$, or $b_k = a_i$ for any $i, j, k$. Similarly, $a_i < b_j$ will always be true for all $i, j$. So the only interesting unpacked atomic formulas $\alpha(\bar{x}, \bar{y})$ that matter are one of the following 3 categories:
\begin{itemize}
\item $x_i + x_j = x_k$, $x_i - x_j = x_k$, $\sigma(x_i) = x_j$, $x_i = x_j$, $x_i < x_j$, $x_i = 0$
\item $y_i + y_j = y_k$, $y_i - y_j = y_k$, $\sigma(y_i) = y_j$, $y_i = y_j$, $y_i < y_j$
\item $y_i - y_j = x_k$
\end{itemize}

The first categories of formulas are trivial in $\bar{y}$. So, let $\gamma^*(\bar{w}) := w_1 = w_1$, and choose the witnesses $\bar{w} = \bar{b}$ for part (a), and $\bar{b} = \bar{w}$ for part (b).

The second categories of formulas are trivial in $\bar{x}$. So, let $\gamma^*(\bar{w})$ be the formulas $w_i + w_j = w_k$, $w_i - w_j = w_k$, $\sigma(w_i) = w_j$, $w_i = w_j$ and $w_i < w_j$ respectively, and choose the witnesses $\bar{w} = \bar{b}$ for part (a), and $\bar{b} = \bar{w}$ for part (b).

Finally, for the third category, let $\gamma^*(\bar{w})$ be the formula $w_i = w_j$ and choose the witnesses $\bar{w} = \overline{0^-b^+}$ for part (a), and $\bar{b} = \overline{b_0^-w^+}$ for part (b) where $\overline{b_0}$ is an witness to $\bar{a}$ being good for $\gamma(\bar{x},\bar{y})$.

Since we have given explicit formulas, we get precisely the same statement if $\gamma(\bar{x}, \bar{y})$ is now a conjunction of unpacked atomic formulas instead.
\end{proof}

Let us now make an easy but very useful observation about $\omega$-sums.
\begin{thm}
Let $\M$ be an elementary extension of an $\omega$-sum $G$. Suppose $nonstandard(\M)\not=\emptyset$ and $0 < \bar{d}\in\M$ is nonstandard. Suppose also that $\gamma(\bar{x})$ is a conjunction of unpacked atomic formulas such that $\M\models\gamma(\bar{d})$. Then
$$\M\models\forall y\exists\bar{z}>y\;\gamma(\bar{z}).$$
\end{thm}
\begin{proof}
For each $\ell\in\omega$, 
$$\M\models\exists\bar{z}>U_\ell\;\gamma(\bar{z}),$$
witnessed by $\bar{d}$. Since $G\preceq\M$, it follows that for every $\ell\in\omega$,
$$G\models\exists\bar{z}>U_\ell\;\gamma(\bar{z}).$$
But $G$ is an $\omega$-sum. In particular, $G$ has no nonstandard elements. Thus,
$$G\models\forall y\exists\bar{z}>y\;\gamma(\bar{z}).$$
By elementarity,
$$\M\models\forall y\exists\bar{z}>y\;\gamma(\bar{z}).$$
\end{proof}

Now, we use Lemma 4.33 to prove a parametric version of Theorem 4.34.
\begin{thm}
Let $G$ be an $\omega$-sum and $G\preceq\M$ with $nonstandard(\M)\not=\emptyset$. Let $0 < c^*\in\M$ be nonstandard and let $\varphi(x, c^*)$ define a strong cut of $\M_{\ge 0}$. Assume $\M\models\theta(\bar{c}, \bar{d})$ where $\theta(\bar{x}, \bar{y})$ is a quantifier-free formula, $\bar{c}\subseteq\varphi(\M, c^*)$ with $c^*\in\bar{c}$, and $\bar{d}$ is nonstandard with $\bar{d} > \varphi(M, c^*)$. Then,
$$\M\models\forall u\exists\bar{b}>u\;\theta(\bar{c},\bar{b}).$$
\end{thm}
\begin{proof}
By Remark 4.30, there is a positive boolean combination $\beta(\bar{x}, \bar{y}, \bar{z})$ of unpacked atomic formulas such that
$$\vdash\theta(\bar{x},\bar{y})\leftrightarrow\exists\bar{z}\beta(\bar{x}, \bar{y}, \bar{z}).$$

Write $\beta(\bar{x},\bar{y},\bar{z})$ as a DNF:
$$\beta(\bar{x},\bar{y},\bar{z}) = \bigvee_{i\le n} \beta_i(\bar{x},\bar{y},\bar{z}),$$
where each $\beta_i(\bar{x},\bar{y},\bar{z})$ is a conjunction of unpacked atomic formulas. Since existential quantification commutes with disjunction, we get that there is some $i\le n$ such that $\M\models\exists\bar{z}\beta_i(\bar{c},\bar{d},\bar{z})$. Let $\bar{e}$ from $\M$ witness this. Thus, $\M\models\beta_i(\bar{c},\bar{d},\bar{e})$. 

By modifying the unpacked atomic formulas appropriately we can assume without loss of generality that $\bar{e} \ge 0$. Split $\bar{e}$ into $\bar{e}_1$ and $\bar{e}_2$ such that $\bar{e}_1\subseteq\varphi(\M,c^*)$ and $\bar{e}_2>\varphi(\M,c^*)$. Either of $\bar{e}_1$ or $\bar{e}_2$ can be empty. Let $\bar{c}' = \bar{c}\bar{e}_1$ and $\bar{d}' = \bar{d}\bar{e}_2$. Thus, $\M\models\beta_i(\bar{c}',\bar{d}')$. By Lemma 4.33, there is a formula $\beta^*_i(\bar{w})$ such that
$$\M\models\exists\bar{w}(\beta_i^*(\bar{w}) \;\wedge\; (\bar{w} > \varphi(M, c^*))).$$

Since $c^*$ is positive and nonstandard and $c^*\in\varphi(\M, c^*)$, it follows that $\bar{w}$ is positive and nonstandard. Therefore, by Theorem 4.34,
$$\M\models\forall u\exists\bar{w}>u\;\beta_i^*(\bar{w}).$$
But, then by Lemma 4.33 again, we have
$$\M\models\forall u\exists\bar{b}'>u\;\beta_i(\bar{c}',\bar{b}').$$
Inverting the splitting of variables as done before, this implies
$$\M\models\forall u\exists\bar{b}>u\;\exists\bar{z}\;\beta_i(\bar{c},\bar{b},\bar{z}).$$

In particular, $\M\models\forall u\exists\bar{b}>u\;\theta(\bar{c},\bar{b}).$
\end{proof}

\begin{defn}
Let $G = \oplus_{i\in\omega}H_i$ be an $\omega$-sum, $0 < a\in G$, $\rho$ an algebraic number, and $L\in\Z[\sigma]$ be a minimal polynomial that $\rho$ satisfies. Then, we say 
$$a \mbox{ is {\em near $\rho$}} \iff a > 0 \;\wedge\; \exists b (a \le b\le 2a \;\wedge\; L(b) = 0).$$
Equivalently, we also say
$$a \mbox{ is {\em near $\rho$}} \iff a > 0 \;\wedge\; \exists b (a \le b\le 2a \;\wedge\; \sigma(b) = \rho\cdot b).$$
\end{defn}

\begin{rmk}
It is easy to see that if $G = \oplus_{i\in\omega}H_i$ is an $\omega$-sum, then for any algebraic $\rho$, the definable set $near_\rho(G) := \{x\in G\mid x\mbox{ is } \near\rho\}$ is a union of convex classes, where each convex class is of the form $(U_{\ell+1}\setminus U_\ell)\cap\{x\in G\mid x > 0\}$ for some $\ell\in\omega$ such that $H_\ell\models\MODDAG_\rho$. Although this nice characterization of the set of elements near $\rho$ usually fails in an elementary extension $\M$ of $G$ because there could be positive nonstandard elements in $\M$ which do not belong to any $U_\ell$ but are still $\near\rho$, the fact that $\near_\rho(\M)$ is a union of convex classes is however preserved by elementarity because of the following:
$$G \models \forall x[x \mbox{ is }\near \rho \to\exists u\exists v(u < x < v \;\wedge\; \forall y(u < y < v\to y \mbox{ is }\near\rho))].$$
\end{rmk}

This motivates the following definition.
\begin{defn}
Let $G = \oplus_{i\in\omega}H_i$ be an $\omega$-sum, and $0 < a\in G$ be $\near\rho$ for some algebraic $\rho$. By the {\em $\rho$-class of $a$}, we mean the largest convex subset $A$ of $G$ containing $a$ with the property that for all $c$ in $A$, $c$ is $\near\rho$. By a {\em $\rho$-class of $G$}, we mean the $\rho$-class of some element $0 < a\in G$. We denote the $\rho$-class of $a$ by $[a]_\rho$. It is easy to see that, for a fixed $\rho$, $[a]_\rho$ is a definable set (with parameter $a$):
\begin{eqnarray*}
b\in[a]_\rho & \iff & [a \le b \;\wedge\; \forall c(a\le c\le b\to c\mbox{ is } \near\rho)] \\
& & \bigvee\;[b \le a \;\wedge\; \forall c(b\le c\le a\to c\mbox{ is } \near\rho)].
\end{eqnarray*}
\end{defn}

With this definition, we make an important observation.
\begin{lem}
Given an $\omega$-sum $G$, an algebraic $\rho$, and elements $0 < a < a'\in G$ both $\near\rho$, the statement ``$[a']_\rho$ is the next $\rho$-class of $G$ after $[a]_\rho$'' is elementary.
\end{lem}
\begin{proof} 
The following formula $next_\rho(x, y)$ defines the given relation:
\begin{eqnarray*}
next_\rho(x, y) & := & (0 < x < y) \;\wedge\; x \mbox{ is }\near\rho \;\wedge\; y\mbox{ is } \near\rho \;\wedge\; ([x]_\rho\not=[y]_\rho)\;\wedge\\
& & \forall z\Big((x\le z\le y\;\wedge\; z \mbox{ is } \near\rho)\to ([z]_\rho = [x]_\rho\vee [z]_\rho = [y]_\rho)\Big).\qedhere
\end{eqnarray*}
\end{proof}

\begin{lem}
Fix an algebraic $\rho$. Suppose $G = \oplus_{i \in\omega} H_i$ is an $\omega$-sum with $Con^G_\rho$ infinite, but not cofinite, in $\omega$. Then $T =Th(G)$ is not model complete.
\end{lem} 
\begin{proof}
For a contradiction, assume that $T$ is model complete. 

Let $\M$ be an elementary extension of $G$ such that $nonstandard(\M)\not=\emptyset$. \newline
Since $Con^G_\rho = \{i\in\omega\mid H_i\models\forall x(\sigma(x) = \rho\cdot x)\}$ is infinite, but not cofinite, we have
$$G\models\forall x(x \mbox{ is } \near\rho \to \exists y(y \mbox{ is } \near\rho\;\wedge next_\rho(x, y))).$$
By elementarity, this property also holds in $\M$. In particular, $\{x\in\M\mid x \mbox{ is } \near\rho\}$ is cofinal in $\M$. So, pick a nonstandard $0 < a\in\M$ such that $a$ is $\near\rho$.

Now consider the formula $\next_\rho(a, x)$ and pick $a'\in\M$ such that $\M\models\next_\rho(a,a')$. Since $T$ is assumed to be model complete, it follows that there is a quantifier-free formula $\theta(\bar{e})$ in the atomic diagram of $\M$ such that
$$T\cup\{\theta(\bar{e})\}\vdash\next_\rho(a,a').$$
Without loss of generality, we may assume $a, a'\in\bar{e}$. Now define the formula $\varphi(x,a)$ as follows:
$$\varphi_\rho(x, a) := (0\le x\le a)\vee (x\in[a]_\rho).$$
Clearly, $\varphi_\rho(M, a)$ defines a strong cut. Partition the variables $\bar{e}$ as $(\bar{c}, a, a', \bar{d})$ such that $\bar{c}\subset\varphi_\rho(M, a)$ and $\bar{d}>\varphi_\rho(M,a)$. So, we have
\begin{eqnarray}
\nonumber
T \cup\{\theta(\bar{c}, a, a', \bar{d})\} & \vdash & \next_\rho(a, a')\\\nonumber
T & \vdash & \theta(\bar{c}, a, a', \bar{d})\to\next_\rho(a, a')\\
T & \vdash & \forall u\forall\bar{v}(\theta(\bar{c}, a, u, \bar{v})\to\next_\rho(a, u))
\end{eqnarray}
Now, $\M\models\theta(\bar{c}, a, a', \bar{d})$ since $\theta(\bar{e})$ is in the atomic diagram of $\M$. Also, $a$, $a'$ and $\bar{d}$ are nonstandard. So by Theorem 4.35, we have
$$\M\models\forall y\exists w\exists\bar{z}(w > y \;\wedge\;\bar{z}>y\;\wedge\;\theta(\bar{c}, a, w, \bar{z})).$$
Pick $y > [a']$ and choose witness $a'' > y$ and $\bar{d}^*> y$ such that 
$$\M\models\theta(\bar{c}, a, a'', \bar{d}^*).$$
Then it follows from (1) that
$$\M\models\next_\rho(a, a''),$$
which is a contradiction because $\M\models\next_\rho(a, a')$ and also $\M\models a'' > [a']$. This proves that $T$ cannot be model complete.
\end{proof}

\begin{lem}
Fix an algebraic $\rho$. Suppose $G = \oplus_{i \in\omega} H_i$ is an $\omega$-sum such that $Con^G_\rho$ is finite, but $Inc^G_\rho$ and $Dec^G_\rho$ are infinite. Then $Th(G)$ is not model complete.
\end{lem} 
\begin{proof}
Observe that since $Con^G_\rho$ is finite, there is a number $N$ such that $\{x\in G\mid x > U_N \mbox{ and } \sigma(x) > \rho\cdot x\}$ and $\{x\in G\mid x > U_N \mbox{ and } \sigma(x) < \rho\cdot x\}$ are both unions of convex sets. Then, for any $a$ in the first set, it makes sense to define the $\rho^>$-class $[a]^>_\rho$ of $a$ as the largest convex set $A$ containing $a$ such that for all $b\in A$, $\sigma(b) > \rho\cdot b$.

Now a very similar proof as that of the last lemma works with the formula $\next_\rho(x, y)$ replaced by the formula
\begin{eqnarray*}
\next^>_\rho(x,y) & := & (0 < x < y) \;\wedge\; (\sigma(x) > \rho\cdot x) \;\wedge\; (\sigma(y) > \rho\cdot y) \;\wedge\; ([x]^>_\rho\not=[y]^>_\rho)\;\wedge\\
& & \forall z\Big((x\le z\le y\;\wedge\;\sigma(z) > \rho\cdot z)\to([z]^>_\rho = [x]^>_\rho\vee[z]^>_\rho = [y]^>_\rho)\Big)
\end{eqnarray*}
and the observation that because of the given hypothesis
$$G\models\forall x(\sigma(x) > \rho\cdot x\to\exists y(\sigma(y) > \rho\cdot y \;\wedge\; next^>_\rho(x, y))).$$
We also need to replace $\varphi_\rho(x, a)$ by the obvious formula
$$\varphi^>_\rho(x, a) := (0\le x\le a)\vee(x\in[a]^>_\rho).\qedhere$$
\end{proof}

We are now ready to prove the main theorem of this subsection.
\begin{proof}[Proof of Theorem 4.28]
If $Con^G_\rho$ is cofinite for some algebraic $\rho$, then $G$ is elementarily equivalent to a finite $n$-sum (with the last ``coordinate'' of $G$ being a model of $\MODDAG_\rho$), and hence $Th(G)$ is model complete in $\L_{OG, \sigma}$.

If $Con^G_\rho$ is infinite, but not cofinite, for some algebraic $\rho$, then $Th(G)$ is not model complete in $\L_{OG, \sigma}$ by Lemma 4.40.

So now we are in the situation when $Con^G_\rho$ is finite for all algebraic $\rho$. Then, as we have seen before, the sets $Inc^G_\rho$ and $Dec^G_\rho$ are unions of convex sets on a tail of $G$ for all algebraic $\rho$. 

Now, if both $Inc^G_\rho$ and $Dec^G_\rho$ are infinite for some algebraic $\rho$, then by Lemma 4.41, $Th(G)$ is not model complete in $\L_{OG, \sigma}$.

Otherwise, we are in the situation when one of $Inc^G_\rho$ or $Dec^G_\rho$ is cofinite for every algebraic $\rho$. This basically means that for every linear difference polynomial $L\in\Z[\sigma]$, there is a natural number $N = N(L)$ such that
$$G\models\forall x(x > U_N\to L(x) > 0)\vee\forall x(x > U_N\to L(x) < 0) =:\varphi_L.$$
This consistent collection $\{\varphi_L\mid L\in\Z[\sigma]\}$ of infinitely many sentences specifies that the type at infinity of $G$ is a model of $\MODDAG_{\rho^*}$ for some non-algebraic $\rho^*$. In particular, $\rho^*$ is transcendental, infinite, or infinitesimally close to an algebraic. Let us denote by $Th(G|_n)$ the elementary theory of the corresponding $n$-sum $\oplus_{i<n}H_i$. By Remark 4.27, it follows that 
$$Th(G) = \bigcup_{n\in\N}Th(G|_n)\cup\{\varphi_L\mid L\in\Z[\sigma]\}$$
is also model complete in $\L_{OG, \sigma}$ in the event that there is a unique behavior at infinity.
\end{proof}

\section{Ordered Field with Increasing Automorphism}
Now we consider the theory $\OF$ of ordered fields in the language of ordered rings $\L_{OR} := \{+, -, \times, 0, 1, <\}$. We denote by $\RCF$ the theory of real closed fields in the same language. Inspired by previous examples, let us consider the case of an ``increasing'' automorphism. But note that if $\sigma$ is a field-automorphism, then $\sigma(1) = 1$, which implies $\sigma$ is identity on $\Z$ and consequently on $\Q$. Since the rationals are dense in the reals, this shows that the set $\R$ of real numbers, considered as a field, has only the trivial automorphism. Moreover, since $\R_{alg}$, the set of real algebraic numbers, is a prime model of $\RCF$, any automorphism of any real closed field behaves as identity when restricted to $\R_{alg}$.  Also note that, if $1 < x < \sigma(x)$, then $0 < \sigma(x^{-1}) = (\sigma(x))^{-1} < x^{-1}$. Since for any nonstandard $x > \R_{alg}$, the elements $x$ and $2x$ have the same type over $\emptyset$ in $\L_{OR}$, we can hope to have an automorphism, which is increasing only on the ``infinite'' elements, i.e., elements $x$ such that $x > \R_{alg}$. Unfortunately this is not a first-order condition. However, we can change it into a first-order statement with the following definition.

\begin{defn}
An automorphism $\sigma$ on $\OF$ is said to be {\it (eventually) increasing} if
$$\exists y\forall x\;(x > y\rightarrow x < \sigma(x)).$$
\end{defn}

In other words, $\sigma$ is increasing on a tail. We denote by $\RCF^+_\sigma$ the theory of $\RCF$ together with an (eventually) increasing automorphism in the language $\L_{OR, \sigma}$. However, as in the group case in Section 3, this restricted class of automorphisms does not work quite well either because we will now show that $\RCF^+_\sigma$ does not have a model companion in $\L_{OR, \sigma}$. But first, we prove the following lemma.

\begin{lem} Every model of $\RCF_\sigma$ can be extended to a model of $\RCF^+_\sigma$.
\end{lem}
\begin{proof}
Let $R\models\RCF$ and $\sigma$ be an automorphism of $R$.

Let $\L( R) = \L_{OR}\cup\{c_r : r\in R\}$, and $T = Th_{\L( R)}( R)$. Let $T^*\supseteq T$ be a Skolemization of $T$ in some language $\L^* \supseteq \L( R)$.

By Ramsey's Theorem and compactness, there exists a model $M\models T^*$ such that there is an $\L^*$-indiscernible sequence $\langle a_i: i\in\Z\rangle$ in $M$ with $a_j < a_i$ for all $j < i$, and $a_i > R$ for all $i\in\Z$. It follows by $\L^*$-indiscernibility that for each $j < i$, we have $a_j^n < a_i$ for all $n < \omega$. Moreover, $R\preccurlyeq_{\L_{OR}}M$. Without loss of generality, we may assume that $M$ is sufficiently saturated and homogeneous.

Let $\tilde{R} = $ Skolem Hull$(R\cup\{a_i : i\in\Z\})$ in $M$. By Tarski-Vaught, $\tilde{R}\preccurlyeq_{\L^*}M$. In particular, $R\preccurlyeq_{\L_{OR}}\tilde{R}$ and $\tilde{R}\models\RCF$.

Extend $\sigma$ to $\bar{\sigma}$ on $R\cup\{a_i : i\in\Z\}$ by defining: $\bar{\sigma}(r ) = \sigma(r )$ for all $r\in R$, and $\bar{\sigma}(a_i) = a_{i+1}$ for all $i\in\Z$. Note that for each $i\in\Z$, the $\L_{OR}-$type of $a_i$ over $R$ is given by
\begin{eqnarray*}
\tp_{\L_{OR}}(a_i/R) := \{\varphi(x, \bar{b}) &:& \varphi(x,\bar{y})\in\L_{OR} \mbox{ with } lg(\bar{y}) = n \mbox{ for some } n\in\omega,\\
&&\bar{b}\in R^n\mbox{ and }M\models\varphi(a_i, \bar{b})\}.
\end{eqnarray*}
By {\it o-minimality} of $\RCF$, this is exactly equal to the following type:
\begin{eqnarray*}
p(x) = \{\varphi(x, \bar{b}) &:& \varphi(x, \bar{y})\in\L_{OR} \mbox{ with } lg(\bar{y}) = n \mbox{ for some } n\in\omega,\\
&&\bar{b}\in R^n\mbox{ and } R\models\exists z\forall x\;(x >z \rightarrow \varphi(x,\bar{b}))\}.
\end{eqnarray*}
The nice thing about $p(x)$ is that $p(x)$ is $\sigma$-invariant: for any $\varphi(x, \bar{y})\in\L_{OR}$ with $lg(\bar{y}) = n$ and $\bar{b}\in R^n$,
\begin{eqnarray*}
\varphi(x,\bar{b})\in p(x) & \iff & R\models\exists z\forall x\;(x >z \rightarrow \varphi(x,\bar{b})) \\
& \iff & R\models\exists z\forall x\;(x >z \rightarrow \varphi(x,\overline{\sigma(b)})) \\
& \iff & \varphi(x, \overline{\sigma(b)})\in p(x)
\end{eqnarray*}
Thus, the map $\bar{\sigma}$ on $R\cup\{a_i : i\in\Z\}$ is partial elementary in $M$. By homogeneity of $M$, it extends to an automorphism, still denoted by $\bar{\sigma}$, of $M$. Let $\tilde{\sigma} = \bar{\sigma}|_{\tilde{R}}$.

Thus, $\tilde{\sigma}$ is $1-1$. Also for any $a\in\tilde{R}$, there exists a term $\tau(x_1, \ldots, x_m, \bar{b})$ over $R$ such that $a = \tau(a_{i_1}, \ldots, a_{i_m}, \bar{b})$ for some $i_1, \ldots, i_m\in\Z$ and $\bar{b}$ from $R$. Let $c = \tau(a_{i_1 - 1}, \ldots, a_{i_m - 1}, \overline{\sigma^{-1}(b)})$. Then, $c\in\tilde{R}$. Since $\bar{\sigma}$ is an $\L_{OR}$-automorphism of $M$, we have, $\bar{\sigma}(c ) = \tilde{\sigma}(c ) = \tau(\tilde{\sigma}(a_{i_1 - 1}), \ldots, \tilde{\sigma}(a_{i_m} - 1), \overline{\tilde{\sigma}(\sigma^{-1}(b))}) = \tau(a_{i_1}, \ldots, a_{i_m}, \bar{b}) = a$. In other words, $\tilde{\sigma}$ is surjective on $\tilde{R}$. Thus, $\tilde{\sigma}$ is an automorphism of $\tilde{R}$.

Finally, we show that $\tilde{\sigma}$ is (eventually) increasing: Since any element $c\in\tilde{R}$ is of the form $c = \tau(a_{i_1}, \ldots, a_{i_m}, \bar{b})$ for some $i_1, \ldots, i_m\in\Z$ and $\bar{b}$ from $R$, it is easy to see that the sequence $\langle a_i : i\in\Z\rangle$ is cofinal in $\tilde{R}$. Thus, for any $c >> R$, in particular for any $c > a_0$, there exists $i\in\Z$ such that $a_{i} \le c < a_{i+1}$. But then, we have $\sigma(c ) \ge \sigma(a_{i}) = a_{i + 1} > c$. Thus,
\begin{eqnarray*}
& \tilde{R} & \models\forall x\;(x > a_0\rightarrow x < \sigma(x)) \\
\mbox{i.e.,} & \tilde{R} &\models\exists y\forall x\;(x > y\rightarrow x < \sigma(x))
\end{eqnarray*}
Hence, $\tilde{\sigma}$ is (eventually) increasing, and so, $(\tilde{R}, \tilde{\sigma})\models\RCF^{+}_\sigma$.
\end{proof}

Now we are ready to prove the theorem.

\begin{thm}
$\RCF^+_\sigma$ does not have a model companion in $\L_{OR, \sigma}$.
\end{thm}
\begin{proof}
Suppose $\RCF^+_\sigma$ has a model companion, say $T_A$. Recall that $\RCF$ is a complete, model complete theory.

As in Lemma 5.2, let $(M, \sigma)$ be a model of $\RCF^+_\sigma$ extending $(\R_{\alg}, \id)$ such that there is an $\L_{OR}$-indiscernible sequence $\langle a_i: i < \omega\rangle$ in $M$ with $\R_{\alg} < a_i$ and $a_j^n < a_i$ for all $n < \omega$ and $j < i < \omega$. Extend $(M, \sigma)$ to a model $(N, \sigma)$ of $T_A$. Without loss of generality, we may assume $(N, \sigma)$ is sufficiently saturated.

Fix any $2 \le k < \omega$ and consider,
\begin{eqnarray*}
p(x) & = & \{a_i < x : i < \omega\} \\
\psi(x) & = & \exists y\;(a_0 < y < x\;\wedge\;\sigma(y) \le y^k)
\end{eqnarray*}

\textbf{Claim:} In $(N, \sigma)$,
\begin{eqnarray*}
(1) & p(x)\vdash\psi(x), & \mbox{ and} \\
(2) & q(x)\not\vdash\psi(x) & \mbox{ for any finite } q(x)\subset p(x).
\end{eqnarray*}

We show $(2)$ first. Suppose $q(x) \subset \{a_i < x : i < L\}$ for some $L < \omega$. Then, $a_L\in q(N)$. By way of contradiction, suppose that $a_L\in\psi(N)$. Let $y_0$ witness $\psi(a_L)$. Thus, we have
$$a_0 < y_0 < a_L \;\wedge\; \sigma(y_0) \le y_0^k.$$
Without loss of generality, there exists $0\le i < L$ such that $a_i < y_0 < a_{i+1}$. Therefore, 
$$a_{i + 2} = \sigma^2(a_i) < \sigma^2(y_0) \le y_0^{k^2} < a_{i+1}^{k^2},$$
which is a contradiction.

Now we prove $(1)$. Let $c\in p(N)$. Then $c > a_i$ for all $i < \omega$. Let $(M', \sigma|_{M'})$ be a countable $\L_{OR, \sigma}$-elementary substructure of $(N, \sigma)$ such that $a_0, c\in M'$. Now consider the following type $q(x)$ over $M'$:
\begin{eqnarray*}
q(x) & = & \{m < x : m\in M' \mbox{ and } m < a_i \mbox{ for some } i < \omega\} \\
&& \bigcup\;\{x < m : m\in M' \mbox{ and } a_i < m \mbox{ for all } i < \omega\}.
\end{eqnarray*}
It is easy to see that any finite subset of $q(x)$ is realized in $M'$ by some $a_i$. Hence, $q(x)$ is finitely consistent and thus consistent. Let $d\in N$ realize $q(x)$. Clearly $d\in N\setminus M'$. Moreover, $d > a_i$ for all $i < \omega$, i.e., $d\in p(N)$, and $d < c^n$ for all $n\in\Z$. Thus, $d^2 > a_i$ for all $i < \omega$, and $d^2 < c^n$ for all $n\in\Z$. In particular, $d^2\not\in M'$. By o-minimality of $\RCF$, $d$ and $d^2$ have the same type over $M'$ as they belong to the same cut of $M'$. Extend $\sigma$ to $M'\cup\{d, d^2\}$ by defining $\sigma(d) = d^2$. This is a partial elementary map and thus there is some $(\tilde{M}, \tilde{\sigma})\models\RCF_\sigma$ extending $(M', \sigma)$ such that $d\in\tilde{M}$ and $\tilde{\sigma}(d) = d^2$. By Lemma 5.2, there is some $(M'', \sigma'')\models\RCF^{+}_\sigma$ extending $(\tilde{M}, \tilde{\sigma})$, and thus also extending $(M', \sigma)$. Now $c, d\in M''$, and so,
\begin{eqnarray*}
(M'', \sigma'') & \models & a_0 < d < c\;\wedge\; \sigma(d) = d^2 \le d^k \\
\mbox{i.e.,}\;\; (M'', \sigma'') & \models & \exists y\;(a_0 < y < c\;\wedge\;\sigma(y) \le y^k)
\end{eqnarray*}
Since $(M', \sigma)$ is a model of $T_A$, it is an existentially closed model of $\RCF^+_\sigma$. Moreover, $(M', \sigma)\subseteq (M'', \sigma'')$. Thus, the formula $(a_0 < y < c \;\wedge\; \sigma(y) \le y^k)$ has a solution in $(M', \sigma)$. Hence, $(M', \sigma)\models\psi(c)$, and therefore, $(N, \sigma)\models\psi(c)$. 
\end{proof}

Question now arises: can we put more restrictions on the automorphism to get a model complete $\L_{OR, \sigma}$-theory extending the theory of ordered fields with an automorphism? The main problem is that there is a nontrivial subfield, namely the fixed field of $\sigma$, that causes a lot of trouble. We leave this as an open problem.
\begin{prob}
Is there any model complete extension of $\RCF_\sigma$ in $\L_{OR, \sigma}$?
\end{prob}

\section{appendix}
We describe here our second construction, namely a quotient construction, that we mentioned in Section 4.2 for producing more model complete theories of ordered abelian groups with an automorphism.

\begin{defn}
Let $G$ be an ordered abelian group with an automorphism $\sigma$, and $H$ be a convex subgroup of $G$ closed under $\sigma$. Let $G/H$ be the quotient group, which we turn into an ordered difference abelian group as follows:
\begin{eqnarray*}
& & G/H := \{g + H \;|\; g\in G\} \\
& & (g_1 + H) + (g_2+ H) := (g_1 + g_2) + H \\
& & -(g+ H) := (-g) + H \\
& & 0 := 0 + H = H \\
& & (g_1 + H) < (g_2 + H) :\iff g_1 < g_2 \mbox{ and } g_2 - g_1\not\in H \\
& & (g_1 + H) = (g_2 + H) :\iff g_2 - g_1\in H\\
& & \sigma(g + H) := \sigma(g) + H
\end{eqnarray*}
\end{defn}

We now prove the following theorem about quotient construction.
\begin{thm}
Let $H$ be a convex subgroup of $G$ closed under $\sigma$. Suppose $H$ is definable in $G$ by the formula $\varphi_H(x)$. Then, for any $\L_{OG, \sigma}$-formula $\psi(\bar{x})$, there is a formula $\psi^G(\bar{x})$ of quantifier rank at least as high as that of $\psi(\bar{x})$ such that for any tuple $\bar{g} = (g_1, g_2, \ldots)$ from $G$,
$$G/H\models \psi(\overline{g + H}) \iff G\models\psi^G(\bar{g}).$$ 
\end{thm}

Before proving this theorem, we prove a corresponding lemma about the terms.
\begin{lem}
For all terms $t(x_1, \ldots, x_n)$, and all tuples $(g_1, \ldots, g_n)$ from $G$,
$$t^{G/H}(g_1 + H, \ldots, g_n + H) = t^{G}(g_1, \ldots, g_n) + H.$$
\end{lem}
\begin{proof}
We prove this by induction on the length of the terms.
\begin{itemize}
\item[(i)] $t = 0: t^{G/H}(g_1 + H, \ldots, g_n + H) = 0 + H = t^{G}(g_1, \ldots, g_n) + H$.
\item[(ii)] $t = v_i : t^{G/H}(g_1 + H, \ldots, g_n + H) = g_i + H = t^{G}(g_1, \ldots, g_n) + H$.
\item[(iii)] $t = -v_i : t^{G/H}(g_1 + H, \ldots, g_n + H) = -(g_i + H) = (-g_i) + H = t^{G}(g_1, \ldots, g_n) + H.$
\item[(iv)] $t = \sigma(v_i) : t^{G/H}(g_1 + H, \ldots, g_n + H) = \sigma(g_i + H) = \sigma(g_i) + H = t^{G}(g_1, \ldots, g_n) + H.$
\item[(v)] $t = v_i + v_j : t^{G/H}(g_1 + H, \ldots, g_n + H) = (g_i + H) + (g_j + H) = (g_i + g_j) + H = t^{G}(g_1, \ldots, g_n) + H.$
\end{itemize}
This proves the base case. For the induction step, note that a general term in this language is basically a linear difference polynomial $L(x_1, \ldots, x_n)\in\Z[\bar{x}, \sigma(\bar{x}), \ldots,$ $\sigma^m(\bar{x})]$ for some $m$. And it follows easily from induction hypothesis that
$$L^{G/H}(g_1 + H, \ldots, g_n + H) = L^{G}(g_1, \ldots, g_n) + H.$$
\end{proof}

This completes the proof of the lemma. Now we prove the theorem.
\begin{proof}
We prove this by induction on the quantifier rank $r$ of $\psi$.

The base case is that of $r = 0$ when $\psi$ is an atomic formula. There are 2 types of atomic formulas in the language $\L_{OG, \sigma}$, namely, $t_1 = t_2$ and $t_1 < t_2$.
\begin{enumerate}
\item $\psi(\bar{x}) := t_1({\bar{x}}) = t_2(\bar{x}) : $ In this case,
\begin{eqnarray*}
& & G/H \models \psi(\overline{g + H}) \\
& \iff & t_1^{G/H}(\overline{g + H}) = t_2^{G/H}(\overline{g + H}) \\
& \iff & t_1^{G}(\bar{g}) + H = t_2^{G}(\bar{g}) + H \mbox{ (by Lemma 6.3)} \\
& \iff & t_2^{G}(\bar{g}) - t_1^{G}(\bar{g}) \in H \\
& \iff & G \models \varphi_H(t_2(\bar{g}) - t_1(\bar{g}))
\end{eqnarray*}
Thus, let $\psi^G(\bar{x}) := \varphi_H(t_2(\bar{x}) - t_1(\bar{x}))$.
\item $\psi(\bar{x}) := t_1(\bar{x}) < t_2(\bar{x}) : $ In this case,
\begin{eqnarray*}
& & G/H \models \psi(\overline{g + H}) \\
& \iff & t_1^{G/H}(\overline{g + H}) < t_2^{G/H}(\overline{g + H}) \\
& \iff & t_1^{G}(\bar{g}) + H < t_2^{G}(\bar{g}) + H \mbox{ (by Lemma 6.3)} \\
& \iff & t_1^{G}(\bar{g}) < t_2^{G}(\bar{g}) \mbox{ and } t_2^{G}(\bar{g}) - t_1^{G}(\bar{g}) \not\in H \\
& \iff & G \models \psi(\bar{g}) \wedge \neg\varphi_H(t_2(\bar{g}) - t_1(\bar{g}))
\end{eqnarray*}
Thus, let $\psi^G(\bar{x}) := \psi(\bar{x})\wedge\neg\varphi_H(t_2(\bar{x}) - t_1(\bar{x}))$.
\end{enumerate}

Now we consider the boolean connectives. Suppose the induction hypothesis holds for $\phi$ and $\theta$, witnessed by the corresponding formulas $\phi^G$ and $\theta^G$.
\begin{enumerate}
\item $\psi := \phi\wedge\theta : $ In this case, take $\psi^G := \phi^G\wedge\theta^G$.
\item $\psi := \neg\phi : $ In this case, take $\psi^G := \neg\phi^G$.
\end{enumerate}

This leaves us with the last case, namely, the case of quantifiers. Identifying $\forall = \neg\exists\neg$, it is enough to deal with the case of a single existential quantifier.

$\psi(\bar{x}) := \exists y\phi(y, \bar{x}) : $ In this case,
\begin{eqnarray*}
& & G/H \models \psi(\overline{g + H}) \\
& \iff & G/H \models \phi(g' + H, \overline{g + H}) \mbox{ for some } g' + H\in G/H \\
& \iff & G \models \phi^G(g', \bar{g}) \mbox{ for some } g'\in G \mbox{ (by induction hypothesis on } \phi)\\
& \iff & G \models \exists y\phi^G(y, \bar{g})
\end{eqnarray*}
Thus, let $\psi^G(\bar{x}) := \exists y\phi^G(y, \bar{x})$.
\end{proof}

As an immediate corollary, we get
\begin{cor}
Let $G$ be a model of $OAG_\sigma$ and $H$ be a convex subgroup of $G$ closed under $\sigma$. Suppose $H$ is definable in $G$. Then,
\begin{itemize}
\item[(1)] If $G$ eliminates quantifiers, then so does $G/H$.
\item[(2)] If $G$ is model complete, then so is $G/H$.
\end{itemize}
\end{cor}

\bibliographystyle{amsplain}
\bibliography{references}

\end{document}